\documentclass[reqno, letterpaper, oneside]{article}
\usepackage{amsmath,amsfonts,amssymb,amsthm}
\usepackage{comment}
\usepackage{bbm}
\usepackage{setspace}
\usepackage{geometry}
\usepackage{enumerate}

\usepackage[colorlinks=false,,linktoc=allbookmarksopen=true,linktocpage,pdftex]{hyperref}
\usepackage[pdftex]{bookmark}
\hypersetup{
	colorlinks=true,       % false: boxed links; true: colored links
	linkcolor=blue,        % color of internal links
	citecolor=blue,        % color of links to bibliography
	filecolor=magenta,     % color of file links
	urlcolor=blue         
}

\usepackage{tabularx} % extra features for tabular environment

\usepackage{mathtools}
\usepackage{cleveref}
\usepackage{enumitem}

\theoremstyle{plain}% default
\newtheorem{theorem}{Theorem}[section]
\newtheorem{lemma}[theorem]{Lemma}
\newtheorem{proposition}[theorem]{Proposition}
\newtheorem{corollary}[theorem]{Corollary}

\theoremstyle{definition}

\newtheorem{definition}[theorem]{Definition}

\newtheorem{remark}[theorem]{Remark}

\newcommand{\refT}[1]{Theorem~\ref{#1}}
\newcommand{\refC}[1]{Corollary~\ref{#1}}

\newcommand{\refL}[1]{Lemma~\ref{#1}}

\newcommand{\refS}[1]{Section~\ref{#1}}

\newcommand{\refP}[1]{Proposition~\ref{#1}}

\newcommand{\refD}[1]{Definition~\ref{#1}}

\newcommand{\incomp}{\rm Incomp}

\newcommand{\cB}{\mathcal{B}}

\newcommand{\cE}{\mathcal{E}}

\newcommand{\cH}{\mathcal{H}}

\newcommand{\cN}{\mathcal{N}}
\newcommand{\cS}{\mathcal{S}}

\newcommand\lrpar[1]{\left(#1\right)}

\newcommand\lrsqpar[1]{\left[#1\right]}

\newcommand\modulus[1]{\left \lvert #1 \right \rvert}

\newcommand\floor[1]{\lfloor#1\rfloor}

\newcommand{\lsim}{\lesssim}
\newcommand{\gsim}{\gtrsim}

\renewcommand{\Pr}{\mathbb{P}}
\newcommand{\E}{\mathbb{E}}
\newcommand{\R}{\mathbb{R}}
\newcommand{\Z}{\mathbb{Z}}

\newcommand{\Var}{\mathrm{Var}}

\newcommand{\eps}{\varepsilon}

\newcommand{\cL}{\mathcal{L}}

\let\OLDthebibliography\thebibliography
\renewcommand\thebibliography[1]{
  \OLDthebibliography{#1}
  \setlength{\parskip}{0pt}
  \setlength{\itemsep}{0pt plus 0.3ex}
}

\newcommand{\bP}{\mathbb{P}}

\newcommand{\cF}{\mathcal{F}}
\newcommand{\bS}{\mathbb{S}}

\newcommand{\dist}{\text{dist}}
\newcommand\lrbra[1]{\left\{#1\right\}}
\newcommand{\Levy}{\mathcal{L}}

\newcommand{\RLCD}[4]{\textbf{RLCD}^{#1}_{#2,#3}(#4)}
\newcommand{\logRLCD}[4]{\textbf{RLogLCD}^{#1}_{#2,#3}(#4)}

\DeclarePairedDelimiterX{\ip}[2]{\langle}{\rangle}{#1, #2}
\DeclarePairedDelimiterX{\norm}[1]{\lVert}{\rVert}{#1}

\title{A distance theorem for inhomogenous random rectangular matrices}
\author{Manuel Fernandez V}
\date{\today}

\begin{document}
\maketitle\begin{abstract}
    Let $A \in \R^{n \times (n - d)}$ be a random matrix with independent uniformly anti-concentrated entries satisfying $\E\norm{A}_{HS}^2 \leq Kn(n-d)$. Let $X \in \R^n$ be a random vector with uniformly anti-concentrated entries. We show that when $1 \leq d \leq \lambda n/\log n$ the distance between $X$ and $H$ satisfies the following small ball probability estimate:
    \[
    \Pr\lrpar{\dist(X,H) \leq t\sqrt{d}}
    \leq (Ct)^{d} + e^{-cn},
    \]
    for some constants $\lambda,c,C > 0$. This extends the distance theorems from \cite{RV}, \cite{GL}, and \cite{LTV} by dropping any identical distribution assumptions about the entries of $X$ and $A$. Furthermore it can be applied to prove numerous results about random matrices in the inhomogenous setting. These include lower tail estimates on the smallest singular value of rectangular matrices and upper tail estimates on the smallest singular value of square matrices.
    \par
    To obtain a distance theorem for inhomogenous rectangular matrices we introduce a new tool for this new general ensemble of random matrices, \textit{Randomized Logarithmic LCD}, a natural combination of the \textit{Randomized LCD}, used in study of smallest singular values of inhomogenous square matrices \cite{LTV}, and of the \textit{Logarithmic LCD}, used in the study of no-gaps delocalization of eigenvectors \cite{no-gaps} and the smallest singular values of Hermitian random matrices \cite{V}.  
\end{abstract}
\section{Introduction}
\par
Given a random matrix $A \in \R^{N \times n}$ one can ask numerous questions about its typical behavior and structure. In non-asymptotic random matrix theory a well-studied problem has been that of determining the behavior of the smallest singular value
\[
\sigma_{n} := \inf_{x \in \bS^{n-1}} \norm{Ax}_2.
\]
For many years sharp estimates for $\sigma_{n}$ in the case of square matrices were limited to that of the standard gaussian matrix, due to Szarek \cite{szarek} and Edelman \cite{edelman}. Subsequent work by Tao and Vu \cite{TaoVu}, Rudelson \cite{Rudelson1} and finally Rudelson and Vershynin  \cite{RV-square},\cite{RV} established the correct behavior for square random matrices with subgaussian entries. Beyond the result itself, the tools developed in these works have been used to prove numerous other quantitative statements about random matrices. Two such examples include upper tail estimates on intermediate singular values \cite{wei},\cite{nguyen} and  delocalization of eigenvectors \cite{no-gaps},\cite{deloc-infty},\cite{deloc-lytova}, all results being for i.i.d. subgaussian matrices. 
\par 
Of course there many situations where random matrix models arise where the matrix is neither i.i.d nor subgaussian. Such situation arise in applied mathematics, such as in the analysis of random networks and in numerical linear algebra.    Once a result has been established for the sub-gaussian case it is natural to try to extend the result to a wider class of random matrix models. Since \cite{RV-square} estimates for the smallest singular value have been extended to various random matrix models. We recommend the recent survey of Tikhomirov \cite{survey} on quantitative invertibility in the non-Hermitian setting for more information on these types of results.
\par
 Alas, it is an ever growing challenge to generalize results about random matrices to these more relaxed settings. One such result is the following: a small ball estimate for the distance between a random vector and a subspace. More precisely, we define the Levy concentration function for a random vector $Z \in \R^n$ as 
 \[
 \cL(Z,t) := \sup_{v \in \R^n} \Pr\lrpar{\norm{Z-v}_2 \leq t},~t \geq 0.
 \]
 Given a random vector $X$ and a subspace $H$ we are interested in an upperbound on the quantity $\cL\lrpar{P_{H^\perp}X,t}$. Here $P_{H^\perp} X$ is understood to mean the orthogonal projection of $X$ onto the orthogonal complement of $H$.
 This type of result, known as a \textit{distance theorem}, is a crucial element in most proofs of optimal lower tail estimates for the  smallest singular value \cite{T},\cite{RV},\cite{GL},\cite{LTV},\cite{campos}, upper tail estimates on the smallest singular value \cite{T}, upper tail estimates on intermediate singular values \cite{nguyen},\cite{wei}, and the condition number \cite{Litvak}, and the delocalization of eigenvectors \cite{no-gaps},\cite{deloc-infty},\cite{deloc-lytova}. 
 In \cite{RV-square},\cite{RV} Rudelson and Vershynin showed that when $H \subset \R^n$ has co-dimension $d$, $1 \leq d \leq \lambda n$ and $X$ and $H$ is the column span of a random matrix $A \in \R^{n \times (n-d)}$ whose entries are i.i.d. subgaussian, then one has
 \begin{equation}\label{eq:dist-1}
\cL(P_{H^\perp}X,t\sqrt{d}) \leq (Ct)^{d} + e^{-cn}.
 \end{equation}
 Here $c,C,\lambda > 0$ depend only on the subgaussian moments of the entries. In \cite{GL} Livshyts showed that the same estimate could be obtained with the subgaussian entry assumption replaced with all entries satisfying a uniform anti-concentration estimate and $A$ consisting of i.i.d. rows. 
\par
In the proofs of previous distance theorems \cite{RV-square},\cite{RV},\cite{GL}, \cite{LTV} the upperbound on $\cL\lrpar{P_{H^\perp} X,t}$ is derived from the lack of `arithmetic structure' of the unit normals of $H$.  In \cite{RV-square},\cite{RV} this lack of arithmetic structure was quantified by a vector having large \textit{essential least common denominator}
which, for a vector $v \in \R^n$ and parameters $L,u > 0$ is defined as 
\begin{equation}\label{eq:LCD-og}
\textbf{LCD}_{L,u}(v) := \inf\{\theta > 0:\dist(\theta v, \Z^n) < \min(u\norm{\theta v}_2,L)\}.
\end{equation}
As an initial observation one can show that, under all previously mentioned settings, unit normals of $H$ are `incompressible' (see  \refD{def:decomposition} and \refL{lem:comp}) with exponentially small failure probability. A result like \refP{prop:incomp} then gives an initial lower bound on \eqref{eq:LCD-og}. By exploiting the fact that the unit normals are orthogonal to many independent random vectors, the initial lower bound can be boosted to a much larger quantity. 
\par
In \cite{V},\cite{no-gaps} the LCD variant \textit{logarithmic least common denominator} was introduced  which, for a vector $v \in \R^n$ and parameters $L,u > 0$ is defined as 
\begin{equation}\label{eq:LogLCD-og}
\textbf{LogLCD}_{L,u}(v) := \inf\left\{\theta > 0 : \dist(\theta v,\Z^n) < L \sqrt{\log_+\lrpar{\frac{u\norm{\theta v}_2}{L}}} \right\}.
\end{equation}
This variant greatly simplifies the argument for deriving upper bounds on $\cL\lrpar{P_{H^\perp}X,t}$ from large LCD in the case where $H^{\perp}$ is multi-dimensional. 
However, both \eqref{eq:LCD-og} and \eqref{eq:LogLCD-og} are difficult to work with in the inhomogenous setting when the entries $X$ or the entries of a column of $A$ are not i.i.d..
Consequently in \cite{LTV}  Livshyts, Tikhomirov, and Vershynin  introduced the LCD variant \textit{randomized least common denominator} which, for a vector $v \in \R^n$ and parameters $L,u > 0$ is defined as 
\begin{equation}\label{eq:RLCD}
    \RLCD{X}{L}{u}{v} := \inf\left 
    \{\theta > 0 : \E\dist^2(\theta(X \star v),\Z^n)  < \min(u\norm{\theta v}_2^2,L^2)
    \right\}.
\end{equation}
Using this variant they proved a distance theorem in the inhomogenous setting where $H$ has co-dimension 1 and proved optimal lower tail estimates for the smallest singular value of square inhomogenous matrices. 
Given their result a natural continuation would be to prove a distance theorem in the inhomogenous setting when $H$ has co-dimension greater than 1. We manage to prove a distance theorem with the same assumption on $X$ and $H$ as in \cite{LTV} but with co-dimension of $H$ now being between 1 and $O(n/\log n)$.
\begin{theorem}\label{thm:distance}
Let $K \ge 1, b \in (0,1)$. There exists constants $c_1,c_2,C,\lambda > 0$, depending on $b$ and $K$ such that the following is true: 
Let $1 \leq d \leq \lambda n/\log(n)$. Let $X \in \R^n$ be a random vector that satisfies $\E\norm{X}_2^2 \le c_1n^2$ and has independent entries $\xi$ satisfying $\sup_{x \in \R}\bP(|\xi-x| < 1) < b$. Let $A \in \R^{n\times (n - d)}$
be a random matrix satisfying $\E\norm{A}_{HS}^2 \le KNn$ with independent entries satisfying $\sup_{x \in \R}\bP(|\xi-x| < 1) < b$. Let $H$ denote the subspace spanned by the columns of $A$. Then
\begin{equation}\label{eq:distance}
\cL(P_{H^\perp}X,t\sqrt{d}) \leq (Ct)^d + e^{-cn}.
\end{equation}
\end{theorem}
For technical reasons we are unable to extend the result to co-dimension up to $O(n)$. However, for certain applications we have in mind, such as lower tail estimates for the smallest singular values estimates of inhomogenous rectangular matrices, \refT{thm:distance} will suffice. Indeed, once  $d$ is of order $n/\log n$ comparatively simpler tools can be used to obtain said estimates. Other applications  of \refT{thm:distance} we have in mind include determining upper tail estimates for the smallest singular value, intermediate singular values, and condition number, all in the inhomogenous setting. Lastly, we note that Rudelson and Vershynin proved a distance theorem in the setting where $H$ is an arbitrary subspace and the entries of $X$ are merely independent and satisfy $\cL(X_i,1) \leq b$ for some $b \in (0,1)$, but with a weaker small ball estimate \cite{distance}.
\par 
Because we seek a distance theorem in the inhomogenous setting with $H$ having co-dimension greater than 1 we define the LCD variant \textit{Randomized Logarithmic LCD} \eqref{eq:RLCD:vector}, which for a vector $v \in \R^n$ and parameters $L,u > 0$ is defined as 
\begin{equation}\label{eq:log-RLCD}
    \logRLCD{X}{L}{u}{v} = \inf \left\{\theta > 0~:~\E\dist^2(\theta(X \star v),\Z^n) < L^2\log_+\lrpar{\frac{u\norm{\theta v}_2}{L}}\right\}.
\end{equation}
In \refS{Sec:RLCD} we prove a number of facts about Randomized Logarithmic LCD that will be necessary in proving \refT{thm:distance}. These include the quantity's stability under pertubations of $v$, a general lower bound on LCD when $v$ is incompressible, and that LCD is decreasing in the parameter $L$.
\par
Turning to the proof of \refT{thm:distance}, we follow a two step approach. The first step is to show that $\cL\lrpar{P_{H^\perp}X,t\sqrt{d}}$ is large when all vectors in $v \in \bS^{n-1} \cap H^\perp$ have large LCD. The second step is to show all vectors $v \in \bS^{n-1} \cap H^\perp$ indeed have large LCD.
To implement the first step using Randomized Logarithmic LCD we mirror the corresponding argument appearing in \cite{no-gaps} for logarithmic LCD. In particular, we first derive a small ball estimate for the product of a matrix and a random vector in terms of the Randomized Logarithmic LCD of the matrix (see \eqref{eq:RLCD:matrix} in \refD{RLCD:def}) and then deduce a small ball estimate for $P_{H^\perp}X$ in terms of the smallest LCD of any unit normal of $H$. This is done in  \refS{Sec:Distance:General} and only slight modifications to the argument in \cite{no-gaps} are required.
Implementing the second step is really the crux of the matter. To do so we adopt the proof strategy used in \cite{LTV} to prove that the unit normals of $H$, with co-dimension 1, have large randomized LCD. Suppose our goal is to prove that LCD of all vectors in $\bS^{n-1} \cap H^\perp$ is at least $\overline{D}$. Then our strategy is as follows:
\begin{enumerate}
    \item \label{LCD:step1}
    Prove that the LCD of all vectors $v \in \bS^{n-1} \cap \cH^\perp$ is at least $\underline{D} > 0$.
    \item \label{LCD:step2}
    Prove that the LCD of all vectors $v \in \bS^{n-1} \cap \cH^\perp$ is at least $\overline{D}$ or less than $\underline{D}$.
\end{enumerate}
Let us first discuss \ref{LCD:step1}.
Recall that we may assume that the unit normals are incompressible. Therefore by \refP{prop:incomp} a large portion (in the $\ell_2$ sense) of the entries of a unit normal are close to 0. This means $\dist(\theta v,\Z^n)$ is of order $\norm{\theta v}_2$ when $\theta$ is not too large which gives a lower bound on the original LCD  and the logarithmic LCD. A similar argument provides a lower bound for randomized LCD \eqref{eq:RLCD}. In the inhomogenous setting this gives a lower bound of order $n/\sqrt{\Var|X|}$. Note, however, that the trivial lowerbound on Randomized Logarithmic LCD is $L/u$. Since a non-trivial upper bound on $\cL(P_{H^\perp}X,t)$ requires that $L$ scale with $\sqrt{d}$ (see \refP{prop:sbp:matrix:general}, \refC{cor:sbp-for-projection}) the previous lower bound in the inhomogenous setting is not sufficient when $\Var|X|$ is at least order $n^2/d$, which can be much smaller than what the second moment assumption guarantees. To get around this we use the probabilistic method to show that for every unit normal there exists a choice of $L$ of order $\sqrt{d}$ such that the Randomized Logarithmic LCD is at least $L/u + cn/\sqrt{\Var|X|}$, for $c > 0$ an absolute constant. Using a rounding argument one can further restrict each vector's choice of $L$ to a small set. Once the LCD of each unit normal, with respect to its choice of $L$, is shown to be large we can deduce that all unit normals have large LCD with respect to a common choice of $L$. 
\par
In contrast to \ref{LCD:step1}, the proof of \ref{LCD:step2} is rather involved. Roughly speaking, we first decompose the subset of $\bS^{n-1} \cap H^\perp$ into level sets according to the quantity $\min_{1 \leq i \leq n-d}\logRLCD{A_i}{L}{u}{v}$ and use a careful recursive argument to show that the level sets in the range $[\underline{D},\overline{D})$ are empty. This information is then used to show that $\logRLCD{X}{L}{u}{v}$ is at least $\overline{D}$ for all $v \in \bS^{n-1} \cap H^\perp$. Although the guarantee is stronger than necessary, it allows us to exploit the fact that $v$ is in the nullspace of $A^\top$. To prove that each level set is empty we rely on a special type of `lattice-like' $\eps$-net over $\bS^{n-1} \cap H^\perp$. More concretely, if the level set contains $v$ then some net point $y$ has $\min_{1 \leq i \leq n-d}\logRLCD{A_i}{L}{u}{y}$ in the same level set as $v$ with $\norm{A^\top y}_2$ being small, which is unlikely. 
We remark that this type of net first appeared in the work of Livshyts \cite{GL} and was a key ingredient in the proof of \ref{LCD:step2} in \cite{LTV}.
\par 
  Because of the technical nature of the proof of \ref{LCD:step2} we defer most of the discussion to \refS{Sec:Distance:Theorem}. Here, however, we mention the following: It turns out that important properties of our $\eps$-net (see \refL{lem:distance:1} and \refT{thm:distance:3}) only hold when $\max_{1 \leq i \leq n-d}\E\norm{A_i}_2^2$ is sufficiently small. Unless $d$ is small,  the second moment assumption is insufficient. Therefore we cannot use the fact that $H^\perp$ is contained in the nullspace of $A^\top$ in our recursive argument. Instead we notice that $H$ contains the span of the submatrix $Q$, where $Q$ consists of all columns of $A_i$ for which $\E\norm{A_i}_2^2$ is of order $n^2/d$. When $d$ is large the condition is stronger than the second moment assumption. Using the Hilbert-Schmidt norm assumption and Markov's inequality it follows that $Q \in \R^{n \times (n-Cd)}$, for an appropriate $C \geq 1$. Since $Q$ only has $(C-1)d$ fewer columns than $A$ and $H^\perp$ is contained in the nullspace of $Q^\top$, we are able to use $Q$ in place of $A$ in the proof of \ref{LCD:step2}.
  \par 
  We end this section with an outline of the rest of the paper. In \refS{Sec:Pre} we discuss some preliminaries and list the assumptions of the inhomogenous setting. In \refS{Sec:RLCD} we introduce the Randomized Logarithmic LCD and prove the relevant properties of it. In \refS{Sec:Distance:General} we reduce the problem of proving a small ball estimate for $\cL\lrpar{P_{H^\perp}X}$ to showing that the Randomized Logarithmic LCD of the unit normals of $H$ are large. In \refS{Sec:Disc} we state the known discretization results necessary to define the $\eps$-net and prove its various properties. Finally in \refS{Sec:Distance:Theorem} we prove that the Randomized Logarithmic LCD of the unit normals of $H$ are large and deduce \refT{thm:distance} as a corollary. We conclude the section with a discussion on the difficulties with extending \refT{thm:distance} to larger co-dimension.  
\section*{Acknowledgements}
The author would like to thank Galyna Livshyts for numerous helpful discussions and feedback. The author was partially supported by a Georgia Tech ARC-ACO Fellowship and NSF-BSF DMS-2247834 while working on this paper. 
\section{Preliminaries}\label{Sec:Pre} In this paper the parameters $b \in (0,1)$ and $K \ge 1$ are taken to be absolute constants. Beyond satisfying their constraints, their choice is arbitrary. However, their value/choice is to be fixed throughout the entire paper. We will use $c,c_1,c_2,\cdots,$ etc. to denote a sufficiently small absolute constant that depends on $b$ and $K$ (recall that $b$ and $K$ are absolute constants) and we will use $C,C_1,C_2,\cdots,$etc. to denote a sufficiently large absolute constant that depends on $b$ and $K$. Furthermore $C$ and $c$ may appear multiple times in the same context with different meaning each time, such as in an inequality chain, so as to simplify expressions and collect terms. 
For notation purposes we write $x \lsim y$ to mean $x/y$ is at most a constant depending only on $K$ and $b$ and write $x \approx y$ to mean $x \lsim y \lsim x$.  
We say $x \ll y$ if $x/y$ is at most a sufficiently small constant depending only on $K$ and $b$. 
Lastly, when speaking about a certain LCD variant we may simply refer to it as LCD when the variant being discussed is clear from the context. 
We now list some assumptions. We say that a random variable $\xi$ is \textit{uniformly anti-concentrated} if $\sup_{u \in \R} \bP[|u - \xi| \leq 1] \leq b$.
We say that a random vector $X \in \R^n$ satisfies the \textit{second moment assumption} if $\E\norm{X}_2^2 \ll n^2$.
We say that a random matrix $A \in \R^{N \times n}$ satisfies the \textit{Hilbert-Schmidt norm assumption} if $\E\norm{A}_{HS}^2 \le KNn$.
Throughout the paper the random vectors and matrices that we consider will have independent entries. 
We now recall the concept of a decomposition of the sphere into `compressible' and `incompressible' vectors. This notion was introduced by Rudelson and Vershynin in \cite{RV-square}, although a similar idea was used in an earlier work by Litvak, Pajor, Rudelson and Tomasz-Jaegerman\cite{Litvak}. 
\begin{definition}\label{def:decomposition}
    Let $\delta,\rho \in (0,1)$. We say that a vector $x \in \bS^{n-1}$ is $(\delta,\rho)$-compressible if there exists a vector $y \in \bS^{n-1}$ with at most $\delta n$ non-zero entries for which $\dist(x,y) \leq \rho$.
We write $Comp(\delta,\rho)$ to denote the set of ($\delta,\rho$)-compressible vectors and define $\incomp(\delta,\rho) := \bS^{n-1}\setminus Comp(\delta,\rho)$.
\end{definition}
There are two important properties about the sphere decomposition that we will use. The first property is that incompressible vectors contain a large subset of coordinates that are `spread'.  
\begin{proposition}[Lemma 2.5 \cite{RV}]\label{prop:incomp}
    Let $x \in \incomp(\delta,\rho)$. Then there exists a subset of indices $J \subset \{1,2,\cdots,n\}$ of size $|J| \geq \frac{1}{2}\rho^2\delta n$ such that 
\begin{equation}\label{eq:spread}
    \frac{\rho}{\sqrt{2n}} \leq |x_j| \leq \frac{1}{\sqrt{\delta n}}~\text{ for all } j \in J.
    \end{equation}
\end{proposition}
The second property is that rectangular random matrices send compressible vectors to non-small vectors. In particular compressible vectors avoid the matrix's null space.
 \begin{lemma}[Lemma 5.3 \cite{GL}]\label{lem:comp}
 Let $A \in \R^{N \times n}$ be a random matrix that satisfies the Hilbert-Schmidt norm assumption and whose entries are uniformly anti-concentrated. Then  
 \begin{equation}\label{eq:comp}
 \bP\lrpar{\inf_{x \in comp(\delta,\rho)} \norm{Ax}_2 \leq c\sqrt{N  }} \leq e^{-cN}.
 \end{equation}
 where $\delta,\rho \in (0,1)$ and $c > 0$ depend only on $K,b$.
 \end{lemma}
 So as to simplify our arguments we  will also fix the choice of $\delta$ and $\rho$, as guaranteed by \refL{lem:comp}, for the remaining sections.

\section{Randomized Logarithmic LCD}\label{Sec:RLCD}
We recall the definition of Randomized LCD as introduced in \cite{LTV}.
\begin{definition}[Definition 2.6 \cite{LTV}]\label{def:RLCD}
Let $X \in \R^n$ be a random vector, let $L > 0, u \in (0,1)$ and let $v \in \R^n$ be a deterministic vector. The Randomized LCD is defined as 
    \[
    \RLCD{X}{L}{u}{v} := \inf\lrbra{\theta > 0~:~\E\dist^2(\theta v \star \bar{X},\Z^n) < \min(L^2,u\norm{\theta v}_2^2)}.
    \]
\end{definition}

 Randomized LCD is a variant of the essential LCD \ref{eq:LCD-og} first introduced in the seminal work of Rudelson and Vershynin on the smallest singular value of square matrices \cite{RV-square}. Because we seek a distance theorem between a random vector and a random subspace with co-dimension greater than 1 we introduce the Randomized Logarithmic LCD, a variant of the Logarithmic LCD \ref{eq:LogLCD-og}, first used in the study of no-gaps delocalization of eigenvectors of random matrices \cite{no-gaps} and the invertibility of symmetric random matrices \cite{V}. 

\begin{definition}\label{RLCD:def}
Let $L > 0, u \in (0,1)$, let $X \in \R^N$ be a random vector and let $A \in \R^{N \times n}$ be a random matrix.

\begin{enumerate}
\item  Given a vector $v \in \R^N$ we defined the Randomized Logarithmic LCD of $v$ as 
\begin{equation}\label{eq:RLCD:vector}
\begin{split}
\logRLCD{X}{L}{u}{v} &:= \inf \lrbra{\theta > 0~:~ \E \dist^2(\theta v \star \bar{X},\Z^N) < L^2 \cdot \log_+\lrpar{\frac{u\norm{\theta v}_2}{L}}}, 
\\
\logRLCD{A}{L}{u}{v} &:= \min_{i} \logRLCD{A_i}{L}{u}{v}.
\end{split}
\end{equation}
\item 
Given a matrix $V \in \R^{n \times N}$, we define the Randomized Logarithmic LCD of $V$ as 
\begin{equation}\label{eq:RLCD:matrix}
\logRLCD{X}{L}{u}{V} := \inf \lrbra{\norm{\theta} > 0~:~ \E \dist^2(V^\intercal \theta \star \bar{X},\Z^N) < L^2 \cdot \log_+\lrpar{\frac{u\norm{V^\intercal \theta}_2}{L}}}. 
\end{equation}
\item 
Given a subspace $E \subseteq \R^N$ we define the Randomized Logarithmic LCD of $E$ as 
\begin{equation}\label{eq:RLCD:subspace}
\logRLCD{X}{L}{u}{E} := \inf_{v \in E \cap \bS^{N-1}}\logRLCD{X}{L}{u}{v}.
\end{equation}
\end{enumerate}
\end{definition} 
We will only use the matrix and subspace versions of the Randomized Logarithmic LCD to show that an upper bound on $\cL\lrpar{P_{H^\perp X},t\sqrt{d}}$ can be obtained from a lower bound on $\inf_{v \in \bS^{n-1} \cap H^\perp} \logRLCD{X}{L}{u}{v}$. We thus defer relevant lemmas about these quantities to \refS{Sec:Distance:General}. For now we will show that the vector version of Randomized Logarithmic LCD satisfies a number of properties. In particular we will show that the LCD of a vector is stable under pertubations, that every incompressible vector has a non-trivial LCD lower bound, and that the LCD is decreasing in the parameter $L$. We will also show that the Randomized Logarithmic LCD can be lower bounded by the Randomized LCD, for an appropriate choice of parameters. 
\subsection{Stability of Randomized Logarithmic LCD}
 Our proof of stability of Randomized Logarithmic LCD is analogous to the proof of stability of Randomized LCD from \cite{LTV}. 

\begin{lemma}[Stability of Randomized Logarithmic LCD for vectors]\label{lem:stability} Let $L > 0, u \in (0,1)$ and let $X \in \R^n$ be a random vector. Let $0 < r_1 < r_2$ be positive parameters and let $x,y$ be points in the annulus $r_2B_2^n \setminus r_1B_2^n$. Fix any tolerance level $\eps > 0$ satisfying 
\begin{equation}\label{eq:stability:assumption:1}
\eps^2 \Var(X) \leq \frac{1}{8}  \frac{L^2}{D^2} \cdot \log_+\lrpar{\frac{u\norm{D x}_2}{L}},
\end{equation}
where $D = \logRLCD{X}{L}{u}{x}$. Then whenever $\norm{x-y}_\infty < \eps$ we have 
\begin{equation}\label{eq:stability:guarantee:1}
\logRLCD{X}{2L}{2u(r_2/r_1)}{y} \leq
D \leq 
\logRLCD{X}{L/2}{(u/2)(r_1/r_2)}{y}.
\end{equation}
In addition if $D' \leq D$ and 
\begin{equation}\label{eq:stability:assumption:2}
\eps^2 \Var(X) \leq \frac{1}{8}  \frac{L^2}{(D')^2} \cdot \log_+\lrpar{\frac{u\norm{D' x}_2}{L}},
\end{equation} then 
\begin{equation}\label{eq:stability:guarantee:2}
D' \leq 
\logRLCD{X}{L/2}{(u/2)(r_1/r_2)}{y}.
\end{equation}
\end{lemma}
\begin{proof}
Note that 
\[
\E\norm{x \star \bar{X} - y \star \bar{X}}_2^2 = \E\sum_{i = 1}^n \bar{X}_i^2(x_y - y_i)^2 < \eps^2 \E\norm{\bar{X}}_2^2 = 2\eps^2\Var(X).
\]
We first prove \eqref{eq:stability:guarantee:1}.
Since $\E \dist^2(\theta(v \star \bar{X}),\Z^n) - L^2\log_+\lrpar{u\norm{\theta v}_2/L}$ is continuous in $\theta$ the definition of Randomized Logarithmic LCD implies that 
\[
\E\dist^2(Dx \star \bar{X},\Z^n) = L^2 \cdot \log_+\lrpar{\frac{u\norm{\theta x}_2}{L}}.
\]
By triangle inequality and the identity $(a+b)^2 \leq 2a^2+2b^2$ we have that
\begin{align*}
\E\dist^2(Dy \star \bar{X},\Z^n) 
&\leq 2\E\dist^2(Dx \star \bar{X},\Z^n) + 2\E\dist^2(Dx \star \bar{X},Dy \star \bar{X}), 
\\
&\leq 
2L^2 \cdot \log_+\lrpar{\frac{u\norm{Dx}_2}{L}} + 2D^2\E\norm{x \star \bar{X} - y \star \bar{X}}_2^2,
\\
&< 
2L^2 \cdot \log_+\lrpar{\frac{u\norm{Dx}_2}{L}} + 4D^2\eps^2\Var(X), \\
&\leq \frac{5}{2}L^2 \cdot \log_+\lrpar{\frac{u\norm{Dx}_2}{L}},  
\\
&\leq (2L)^2 \cdot \log_+\lrpar{\frac{2u(r_2/r_1)\norm{Dy}_2}{(2L)}}.
\end{align*}
By inspection of \refD{RLCD:def} we conclude that   
\[
\logRLCD{X}{2L}{2u(r_2/r_1)}{y} \leq D.
\]
Similarly, by the triangle inequality and the identity $(a+b)^2 \geq a^2/2 - 2b^2$ we have that
\begin{align*}
\E\dist^2(Dy \star \bar{X},\Z^n)
&\geq \frac{1}{2}\E\dist^2(Dx \star \bar{X},\Z^n) - 2\E\dist^2(Dx \star \bar{X},Dy \star \bar{X}),
\\
&\geq \frac{1}{2}L^2 \log_+\lrpar{\frac{u\norm{D x}_2}{L}} - 2D^2\E\norm{x \star \bar{X} - y \star \bar{X}}_2^2,  
\\
&> \frac{1}{2}L^2 \log_+\lrpar{\frac{u\norm{D x}_2}{L}} - 2D^2\eps^2 \Var(X), 
\\
&\ge \frac{1}{4}L^2 \log_+\lrpar{\frac{u\norm{D x}_2}{L}}, 
\\
&\geq \lrpar{\frac{L}{2}}^2\log_+\lrpar{\frac{(u/2)(r_1/r_2)\norm{D y}_2}{(L/2)}}.
\end{align*}
By inspection of \refD{RLCD:def} we conclude that   
\[
D \leq \logRLCD{X}{L/2}{(u/2)(r_1/r_2)}{y}.
\]
We now prove \eqref{eq:stability:guarantee:2}. By inspection of \refD{RLCD:def} we have, for $0 \leq \theta \leq D' \leq D$, that
\[
\E\dist^2(D'x\star \bar{X},\Z^n) \geq L^2\log_+\lrpar{\frac{u\norm{\theta x}_2}{L}}.
\]
The proof of \eqref{eq:stability:guarantee:2} is now exactly the same as the proof of the lowerbound for \eqref{eq:stability:guarantee:1}. 
\end{proof}
\subsection{Lower bounds on Randomized Logarithmic LCD}

We will give two types of lower bounds on $\logRLCD{X}{L}{u}{v}$. The first type of lower bound is analogous to the lower bound for Randomized LCD appearing in \cite{LTV}. This bound is of order $n/\sqrt{\Var(X)}$. Although generally worse than the second type of lower bound it does not depend on the parameter $L$, and the bound is as strong as the second type when $\Var(X)$ is of order $n^2/L^2$. The second type of lower bound is of the form $L/u + cn/\sqrt{\Var(X)}$, where $c$ is a positive constant. This lower bound is useful for the case where $\Var(X)$ is of order at least $n^2/L^2$. That said it requires that $L$ be changed up to a constant multiplicative factor.
\begin{lemma}\label{lem:RLCD:lb}
There exists $n_0 = n_0(b,K), u = u(b,K)$ such that the following is true: Let $n \geq n_0, L \ge 1$, and let $x \in \incomp(\delta,\rho)$. Let $X \in \R^n$ be a random vector with uniformly anti-concentrated entries. Then
\begin{enumerate}
\item 
\begin{equation}\label{eq:RLCD-lb-1}
    \logRLCD{X}{L}{u}{x} \gsim \frac{n}{\sqrt{\Var(X)}},
\end{equation}
\item 
If $L^2 \ll n$ then there exists $\tilde{L} \in [L,2L]$ such that 
\begin{equation}\label{eq:RLCD-lb-2}
\logRLCD{X}{\tilde{L}}{u}{x} - \frac{\tilde{L}}{u} \gsim \frac{n}{\sqrt{\Var(X)}}.  
\end{equation}
\item 
If $X$ also satisfies the second moment assumption then $\tilde{L}$ can be chosen to come from 
\begin{equation}\label{eq:RLCD-lb-3}
\cN := \{L,2L\} \cup \{L + i/10~:~1 \leq i \leq \floor{10L}\}.  
\end{equation}
\end{enumerate}
\end{lemma}
\begin{proof}
We first prove \eqref{eq:RLCD-lb-1}. Let $S_1$ denote the set of indices guaranteed by \refP{prop:incomp}. Since the entries of $X$ are independent and uniformly anti-concentrated and $|S_1| \gsim n$, by taking $n_0$ sufficiently large Chernoff's inequality implies that, with failure probability at most 1/100, at most $b|S_1|/2$ indices of $S_1$ satisfy $|\bar{X}_i| < 1$. Furthermore by Markov's inequality at most $b|S_1|/4$ indices of $S_1$ satisfy $|\bar{X}_i|^2 \gg \Var(X)/|S_1|$, with failure probability at most $1/100$. We conclude that, with failure probability at most 2/100, $S_1$ contains a subset $S_2$ of size at least $b|S_1|/4 \gsim n$ whose indices satisfy $1 \le |\bar{X}_i| \lsim \sqrt{\Var(X)/n}$.
Since $S_2$ is a subset of $S_1$ it's entries also satisfy $1/\sqrt{n} \lsim |\bar{X}_i \cdot x_i| \lsim  \sqrt{\Var(X)/n}$.
Furthermore for $0 \leq \theta \ll \frac{n}{\sqrt{\Var(X)}}$ the indices of $S_2$ satisfy $|\theta(\bar{X}_i \cdot x_i)| \leq 1/2$.
In particular the closest integer to $\theta(\bar{X}_i \cdot x_i)$ is 0. Letting $\bar{\E}$ denote the expectation conditioned on the existence of $S_2$ it follows that 
\[
\E\lrpar{\dist^2(\theta (\bar{X} \star x),\Z^n)} \geq \frac{98}{100}\bar{\E}\lrpar{\dist^2(\theta (\bar{X} \star x),\Z^n)} \gsim \frac{49}{50} \cdot n \cdot \frac{\theta^2}{n} \gsim \theta^2.
\]
Note now that $\lim_{p \to 0^+}\min_{x > 0}(x^2/\log_+(px)) = \infty$. This implies that $L^2\log_+(\theta u/L) \ll \theta^2$, for $u$ sufficiently small, for all $\theta >0$. Therefore
$\E(\dist^2(\theta(\bar{X}\star x),\Z^n)) \gg L^2\log_+(\theta u/L)$ for all $0 \le \theta \ll n/\sqrt{\Var(X)}$. By inspection of \refD{def:RLCD} we conclude \eqref{eq:RLCD-lb-1}.
\par
Next, we prove \eqref{eq:RLCD-lb-2}. Since we have already proven \eqref{eq:RLCD-lb-1} we will assume that $L \gsim n/\sqrt{\Var(X)}$.  Pick $\tilde{L}$ uniformly from $[L,2L]$ and independently of $X$. We first bound 
 \[
 \E_{\bar{X}}\E_{\tilde{L}} \min_{\theta \in \lrsqpar{\frac{\tilde{L}}{u},\frac{\tilde{L}}{u} + \frac{c_1n}{\sqrt{\Var(X)}}}}\dist^2(\theta (X_i \cdot x),\Z). 
 \]
Let $I(t) := \lrsqpar{\frac{t\tilde{L}}{u}, \frac{t\tilde{L}}{u} + \frac{tc_1n}{\sqrt{Var(X)}}}$. Note that, for all indices in $S_2$, the length of $I(\bar{X}_i \cdot x_i)$ is at most  $1/10$ for $c_1$ sufficiently small. 
We now case on $|(L/u)(\bar{X}_i\cdot x_i)|$.
\begin{enumerate}
    \item $|(L/u)(\bar{X}_i \cdot x_i)| \leq 4/10$: In this case $|2(L/u)(\bar{X}_i \cdot x_i)| \leq 8/10$. Therefore the distance between $I(\bar{X}_i\cdot x_i)$ and $\Z$ is achieved by either the left or right endpoint of $I(\bar{X}_i \cdot x_i)$. In particular the distance is at least $\min(1/10,|(\tilde{L}/u)(\bar{X}_i \cdot x_i)|)$, so 
\[
 \E_{\tilde{L}} \min_{\theta \in \lrsqpar{\frac{\tilde{L}}{u},\frac{\tilde{L}}{u} + \frac{c_1n}{\sqrt{\Var(X)}}}}\dist^2(\theta (X_i \cdot x),\Z)
 \geq \frac{98}{100} \cdot \min\left(\frac{1}{100},(L/u)^2(\bar{X}_i \cdot x)^2\right)  \gsim \min\lrpar{1,\frac{L^2}{n}}.
\]
\item $|(L/u)(\bar{X}_i \cdot x_i)| \ge 4/10$:  Let $J := \lrsqpar{(L/u)(\bar{X}_i \cdot x_i),2(L/u)(\bar{X}_i \cdot x_i)}$. If $(\tilde{L}/u)(\bar{X}_i\cdot x_i)$ is contained in $\cup_{k \in \Z} (k+1/10,k+8/10)$ then the distance between $I(\bar{X}_i \cdot x_i)$ and $\Z$ is at least $1/10$. Since $J$ is an interval we have that  
$|J \cap \lrpar{\cup_{k \in \Z} (k+1/10,k+8/10)}| \geq (7/10)(|J| - 3/10)$.
Since $(\tilde{L}/u)(\bar{X}_i \cdot x_i)$ is distributed uniformly over $J$ we conclude that with probability at least
$(7/10)(|J|-3/10)/|J| \geq 7/10 - 21/40 = 7/40$,
the distance between $I(\bar{X}_i \cdot x_i)$ and $\Z$ is at least 1/10, so 
\[
 \E_{\tilde{L}} \min_{\theta \in \lrsqpar{\frac{\tilde{L}}{u},\frac{\tilde{L}}{u} + \frac{c_1n}{\sqrt{\Var(X)}}}}\dist^2(\theta (X_i \cdot x),\Z)\geq \frac{7}{40} \cdot \frac{1}{100} = \frac{7}{4000} \gsim 1.
\]
\end{enumerate}
Since the cases constitute all possible values of $|(L/u)(\bar{X}_i\cdot x_i)|$ we conclude that 
\[
 \E_{\tilde{L}} \min_{\theta \in \lrsqpar{\frac{\tilde{L}}{u},\frac{\tilde{L}}{u} + \frac{c_1n}{\sqrt{\Var(X)}}}}\dist^2(\theta (X_i \cdot x_i),\Z) \gsim \min\lrpar{1, \frac{L^2}{n}}.
\]
Since $|S_2| \gsim n$ with probability at least 98/100 an application of linearity of expectation gives 
\begin{align}
\E_{\bar{X}}\E_{\tilde{L}} \min_{\theta \in \lrsqpar{\frac{\tilde{L}}{u},\frac{\tilde{L}}{u} + \frac{c_1n}{\sqrt{\Var(X)}}}}\dist^2(\theta (X \star x),\Z^n) 
&\geq 
\E_{\bar{X}}\sum_{i \in S_2}\E_{\tilde{L}} \cdot \min_{\theta \in \lrsqpar{\frac{\tilde{L}}{u},\frac{\tilde{L}}{u} + \frac{c_1n}{\sqrt{\Var(X)}}}} \dist^2\lrpar{\theta(X_i \cdot x_i),\Z^n}, 
\\
&\gsim n \cdot \min(1,L^2/n) = \min(n,L^2) = L^2,
\end{align}
with the last inequality following from the assumption that $L^2 \ll n$.
Tonelli's theorem then gives
\[ 
\E_{\tilde{L}}\E_{\bar{X}} \min_{\theta \in \lrsqpar{\frac{\tilde{L}}{u},\frac{\tilde{L}}{u} + \frac{c_1n}{\sqrt{\Var(X)}}}}
\dist^2(\theta (X \star x),\Z^n) =
\E_{\bar{X}}\E_{\tilde{L}} \min_{\theta \in \lrsqpar{\frac{\tilde{L}}{u},\frac{\tilde{L}}{u} + \frac{c_1n}{\sqrt{\Var(X)}}}} \dist^2(\theta (X \star x),\Z^n) \gsim L^2.
\]
Therefore, by the Probabilistic Method, there exists a deterministic choice of $\tilde{L} \in [L,2L]$ for which 
\[
\E \min_{\theta \in \lrsqpar{\frac{\tilde{L}}{u},\frac{\tilde{L}}{u} + \frac{c_1n}{\sqrt{\Var(X)}}}}\dist^2(\theta (X \star x),\Z^n)  \gsim L^2.
\]
Now note that $\log_+(\theta u/\tilde{L}) \ll 1$ when $\theta \le \tilde{L}e^{c_2}/u$ and $c_2$ is sufficiently small. Since we assume that $\tilde{L} \gsim n/\sqrt{\Var(X)}$ we have that $\min(\tilde{L}(e^{c_2}-1)/u,c_1n/\sqrt{\Var(X)}) = c_3n/\sqrt{\Var(X)}$ for $c_3$ sufficiently small.
In particular 
\[
\E \min_{\theta \in \lrsqpar{\frac{\tilde{L}}{u},\frac{\tilde{L}}{u} + \frac{c_3n}{\sqrt{\Var(X)}}}}\dist^2(\theta(\bar{X} \star x),\Z^n) \gg \tilde{L}^2\log_+(\theta u/\tilde{L}).
\]
 By inspection of \refD{def:RLCD} we thus conclude \eqref{eq:RLCD-lb-2}.
\par
We now prove the third claim. Take $\ell$ to be the smallest element in $\cN$ such that $\tilde{L} \leq \ell$. Since $X$ satisfies the second moment assumption we have that $n/\sqrt{\Var(X)} \gg 1$. Therefore we may assume that $c_3n/\sqrt{\Var(X)}\ge 1/(5u) \ge 2(\ell-L)/u$. It follows that
\[
\E_{\bar{X}}\min_{\theta \in \lrsqpar{\frac{\ell}{u},\frac{\ell}{u} + \frac{(c_3/2)n}{\sqrt{\Var(X)}}}}\dist^2(\theta (X \star x),\Z^n) \gg \ell^2\log_+(\theta u/\ell).
\]
By inspection of \refD{def:RLCD} we conclude \eqref{eq:RLCD-lb-2}.
\end{proof}
\subsection{Randomized Logarithmic LCD decreases with $L$}
For a certain technical argument that we will use later on (see \refP{prop:distance:2}) we will need to lower bound $\logRLCD{X}{1}{u}{v}$ by $\logRLCD{X}{L}{u}{v}$ when $L \geq 1$. We show that such a lower bound holds so long as $\logRLCD{X}{L}{u}{v}$ is sufficiently large. 
\begin{lemma}\label{lem:RLCD:compare}
Let $L_1 > L_2 > 0$ and $u  \in (0,1)$. Let $X \in \R^n$ be a random vector and let $x \in \R^n$. If 
\begin{equation}\label{eq:RLCD:comp:1}
\logRLCD{X}{L_2}{u}{x} \geq \max\left\{\frac{L_1}{u}, \frac{1}{u} \cdot \frac{L_1^{L_1^2/(L_1^2-L_2^2)}}{L_2^{L_2^2/(L_1^2-L_2^{2})}}\right\}, 
\end{equation}
then 
\[
\logRLCD{X}{L_2}{u}{x} \geq \logRLCD{X}{L_1}{u}{x}.
\]
\end{lemma}
\begin{proof}
Assume for the sake of contradiction that 
$\logRLCD{X}{L_1}{u}{x} > \logRLCD{X}{L_2}{u}{x}$. By the definition of Randomized Logarithmic LCD it follows that there exists $\theta > 0$ such that $\logRLCD{X}{L_2}{u}{x} \leq \theta < \logRLCD{X}{L_1}{u}{x}$ and 
\[
L_1^2\log_+\lrpar{\frac{\theta u}{L_1}} \leq \E\dist^2(\theta(x \star X),\Z^n) < L_2^2\log_+\lrpar{\frac{\theta u}{L_2}}. 
\]
In particular 
\[
L_1^2\log_+\lrpar{\frac{\theta u}{L_1}} < L_2^2\log_+\lrpar{\frac{\theta u}{L_2}}.
\]
By assumption $\theta \geq L_1/u$ and so 
\[
L_1^2\log\lrpar{\frac{\theta u}{L_1}} < L_2^2\log\lrpar{\frac{\theta u}{L_2}}.
\]
Observe now that 
\[
\theta \geq \frac{1}{u} \cdot \frac{L_1^{L_1^2/(L_1^2-L_2^2)}}{L_2^{L_2^2/(L_1^2-L_2^2)}}
\implies \lrpar{\frac{\theta u}{L_1}}^{L_1^2} \geq \lrpar{\frac{\theta u}{L_2}}^{L_2^2}
\implies L_1^2\log\lrpar{\frac{\theta u}{L_1}} \geq L_2^2\log\lrpar{\frac{\theta u}{L_2}}.
\]
Therefore
\[
L_1^2\log_+\lrpar{\frac{\theta u}{L_1}} < L_2^2\log_+\lrpar{\frac{\theta u}{L_2}} \leq L_1^2\log_+\lrpar{\frac{\theta u}{L_1}},
\]
which is a contradiction. Thus no such $\theta$ exists and we conclude that $\logRLCD{X}{L_1}{u}{x} \leq \logRLCD{X}{L_2}{u}{x}$.
\end{proof}
As a corollary we give a simple criteria for deriving lower bounds on $\logRLCD{X}{1}{u}{x}$ from $\logRLCD{X}{L}{u}{x}$ 

\begin{corollary}\label{cor:RLCD:compare}
    Let $L > 1$ and $u \in (0,1)$. Let $X \in \R^n$ be a random vector and let $x \in \R^n$.
    If
    \begin{equation}\label{eq:RLCD:comp:2}
        \logRLCD{X}{1}{u}{x} \ge 2L/u
    \end{equation}
    then $\logRLCD{X}{1}{u}{x} \ge \logRLCD{X}{L}{u}{x}$.
\end{corollary}
\begin{proof}
    By \refL{lem:RLCD:compare} it suffices to satisfy \eqref{eq:RLCD:comp:1}. To reduce the condition to satisfying \eqref{eq:RLCD:comp:2} it suffices to show that $L^{L^2/(L^2 - 1)} \le 2L$, or $L^{1/(L^2-1)} < 2$. However this is immediate since the map $x \mapsto \log(x)/(x^2-1)$ is decreasing on the interval $(1,\infty)$ and $\lim_{x \to 1^+} x/\log(x^2 - 1) = \lim_{x \to 1^+} 1/(2x^2) = 1/2$, meaning $L^{1/(L^2-1)} \le e^{1/2} < 2$.  
\end{proof}

\subsection{Randomized LCD lowerbounds Randomized Logarithmic LCD}
Finally, we show that Randomized LCD is a lower bound on Randomized Logarithmic LCD for an appropriate choice of LCD parameters and range of $\theta$. This allows us to deduce certain facts about Randomized Logarithmic LCD that were originally prove in \cite{LTV} for Randomized LCD (see for instance \refT{thm:unstructured} and \refC{cor:unstructured-log}). 

\begin{lemma}\label{lem:RLCD-vs-logRLCD} 
Let $X \in \R^n$ be a random vector, let $v \in \R^n$ be deterministic, let $\gamma,L > 0$ and let $s,t \in (0,1)$ satisfy $s^2 \leq 2t$. Then
\[
\logRLCD{X}{L}{u}{v} \geq \min\lrpar{\RLCD{X}{a\sqrt{n}}{t}{v}, \frac{L}{u\norm{v}}e^{n(a/L)^2}}.
\]    
\end{lemma}
\begin{proof}
    For convenience we define $w := \norm{\theta v}$. By inspection of \refD{def:RLCD} and \refD{RLCD:def} it suffices to show that 
    $\min(\gamma^2n,tw^2) \geq L^2\log_+(uw/L)$
    for $0 < w \leq (L/u)e^{n(\gamma/L)^2}$. Since the right hand side is 0 for $0 < w \leq L/u$ we may further assume that $w > L/u$. We now case on $w$:
    \begin{itemize}
        \item $w \geq \gamma\sqrt{n/t}$: Then
        $\min(\gamma^2n,tw^2) = \gamma^2n$ and $w \leq  (L/u)e^{n(\gamma/L)^2} \implies  \gamma^2n \ge L^2\log_+\lrpar{uw/L}.
        $
        \item $w < \gamma\sqrt{n/t}$: Then $\min(\gamma^2n,tw^2) = tw^2$. Define $f(w) :=  tw^2 - L^2\log_+(w u /L)$. Note that $f(L/u) > 0$ and $f'(w) = 2tw - L^2/w$ on $[L/u,\infty)$ (for $f'(L/u)$ we mean the right derivative). Since $f'$ is increasing on $[L/u,\infty)$ and $f'(L/u) = (L/u)(2t-u^2) \ge 0$ we conclude that $f(w) > 0$ and $tw^2 \geq L^2\log_+(w u/L)$ for $w > 0$.
    \end{itemize}
\end{proof}

\section{Distance to general subspaces}\label{Sec:Distance:General}
The distance theorem gives an upper bound on the small ball probability for the distance between a random vector $X$ and a subspace $H$. Recall that to prove the distance theorem our first objective is to show that an upper bound on $\cL\lrpar{P_{H^\perp}X,t\sqrt{d}}$ follows from a lower bound on $\inf_{v \in \bS^{n-1} \cap H^\perp}\logRLCD{X}{L}{u}{v}$. To do this we will mirror the arguments appearing in Section 7 of \cite{no-gaps}. The starting point is the observation that, conditioned on the matrix $A$, $P_{H^\perp}X$ is the product of a deterministic projection matrix $P_{H^\perp}$ and a random vector $X$. Write $P_{H^\perp} = UU^\top$.  Using a well-known anti-concentration tool, Esseen's inequality \cite{Esseen}, we will upper bound $\cL\lrpar{UX,t\sqrt{d}}$ via a lower bound on $\logRLCD{X}{L}{u}{U}$. 
\begin{proposition}\label{prop:sbp:matrix:general}
 Let $X \in \R^n$ be a random vector. Let $U \in \R^{n \times d}$ have rank $d$. Denote $D := \logRLCD{\xi}{L}{u}{U}$ and let $t_0 = \sqrt{d}/D$. Then if
 $2L^2\geq d+2$ we have  
\[
\mathcal{L}(U^\top\xi,t\sqrt{d}) \leq  
\lrpar{\frac{CL}{\sqrt{d}u}}\frac{(\max\left\{t,t_0 \right\})^d}{\det(U^\top U)^{1/2} },
\]
where $C > 0$ is an absolute constant.
\end{proposition}
\begin{proof}
 Since the Levy-concentration function is an increasing function in $t$ we may assume that $t \ge t_0$. We recall the $d$-dimensional version of Esseen's inequality \cite{Esseen}. For a random vector $Y \in \R^d$ and parameter $t > 0$ the small ball probability $\cL(Y,\sqrt{d})$ satisfies 
\begin{equation}\label{eq:Esseen}
\mathcal{L}(Y,\sqrt{d}) \leq C_1^d\int_{B(0,\sqrt{d})} \left \lvert \E \exp(2\pi i \ip{\theta}{Y})\right \rvert~d\theta,
\end{equation}
where $|\cdot|$ denotes the modulus and $C_1 > 0$.
Since $\ip{\theta}{U^\top X} = \ip{(U\theta)^\top}{X} = \sum_{j = 1}^n \ip{U^j}{\theta}X_j$ and entries of $X$ are independent we have that 
\begin{equation}\label{eq:esseen-reduce-1}
\begin{split}
    \cL(U^\top X,t\sqrt{d}) &= \cL(U^\top X/t,\sqrt{d}), \\
    &\le C_1^d\int_{B(0,\sqrt{d})} \modulus{\E\exp(2\pi i \ip{\theta/t}{U^\top X})}~d\theta, \\
    &= C_1^d \int_{B(0,\sqrt{d})} \prod_{j = 1}^n |\E\exp(2\pi i \ip{U^j}{\theta/t}X_j)|~d\theta.
\end{split}
\end{equation}
Letting $\tilde{X}$ denote an independent copy of $X$ and writing $\bar{X} := X - \tilde{X}$ one has, for any $s \in \R$, 
\begin{align*}
    \modulus {\E\exp(isX_j)}^2 = \E\exp(isX_j)\overline{\E\exp(is\tilde{X}_j)} = \E\exp(is\bar{X}_j) = \E \cos(s\bar{X}_j),
\end{align*}
with the last equality following from the fact that $\tilde{X}_j$ is a symmetric random variable.

Note now that $x \le \exp((x^2-1)/2)$ for  $x \in [0,\infty)$ and $\cos(2\pi x) - 1 \le -8\cdot \dist^2(x,\Z)$ for $x \in \R$. Since $\modulus {\E\exp(isX_j)}^2 = \E \cos(s\bar{X_j})$ we deduce that 
\begin{equation}\label{eq:char-dist}
|\E\exp(isX_j)| \le \exp\lrpar{\frac{\E\cos(s\bar{X}_j) - 1}{4}} \le \exp\lrpar{-2\E\dist^2(s\bar{X}_j/2\pi,\Z))}.
\end{equation}
Applying \eqref{eq:char-dist} to the last line of \eqref{eq:esseen-reduce-1} with $s = 2\pi\ip{U^j}{\theta/t}$ and using the second equality of \eqref{eq:char-dist} yields
\begin{equation}\label{eq:esseen-reduce-2}
C_1^d \int_{B(0,\sqrt{d})} \prod_{j = 1}^n |\E\exp(2\pi i \ip{U^j}{\theta/t}X_j)|~d\theta \le
C_1^d \int_{B(0,\sqrt{d})} \exp\lrpar{-2\E\dist^2(U^\top (\theta/t) \star \bar{X},\Z^n)}~d\theta.
\end{equation}
Since $\norm{\theta}_2 < \sqrt{d}$ and $\norm{\theta/t}_2 \le D\norm{\theta}_2/\sqrt{d} < D$  \eqref{eq:RLCD:matrix} of \refD{RLCD:def} implies that
\begin{equation}\label{eq:esseen-reduce-3}
C_1^d \int_{B(0,\sqrt{d})} \exp\lrpar{-2\E\dist(U^\top (\theta/t) \star \bar{X},\Z^n)}~d\theta \le C_1^d\int_{B(0,\sqrt{d})}\exp\lrpar{-2L^2 \cdot \log_+ \lrpar{\frac{u\norm{U^\top \theta}_2}{tL}}}.
\end{equation}
To estimate the right hand side of \eqref{eq:esseen-reduce-3}, note that we may write the integrand as $f(J\theta)$, where $J = (u/(tL))U^\top$ and $f(z) := -2L^2\log_+(\norm{z}_2)$. Since $\det(JJ^\top)^{1/2} = (u/(tL))^d\det(U\top U)^{1/2}$ the corresponding change of variable formula gives
\begin{align*}
C_1^d\int_{B(0,\sqrt{d})}\exp\lrpar{-2L^2\cdot \log_+\lrpar{\frac{u\norm{U^\top \theta}_2}{tL}}}~d\theta 
&= \frac{C_1^d}{\det(JJ^\top)^{1/2}}\int_{J(B(0,\sqrt{d}))} \exp(-2L^2 \cdot \log_+(\norm{z}_2))~dz, \\
&\le 
\frac{1}{\det(U^\top U)^{1/2}}\lrpar{\frac{C_1tL}{u}}^d \int_{\R^n} \exp(-2L^2 \cdot \log_+(\norm{z}_2))~dz,
\end{align*}
with the change in the domain in the last inequality valid because the integrand is non-negative.
We now write $\R^n = B(0,1) \cup B(0,1)^c$. The integral of $\exp(-2L^2\log_+(\norm{z}_2^2))$ over $B(0,1)$ is at most its volume which, from standard estimates, is at most $(C_3/\sqrt{d})^d$. For $B(0,1)^c$ we use polar coordinates $(r,\phi)$, noting that $dz = r^{d-1}drd\phi$. In addition, on $B(0,1)^c$ the integrand is simply $\norm{z}_2^{-2L^2}$. Therefore
\begin{align*}
\int_{B(0,1)^c}\norm{z}_2^{-2L^2} &= \int_1^\infty \int_{\bS^{n-1}} r^{-2L^2 +d-1}drd\phi, \\
&= \text{surface area of }\bS^{n-1} \times \int_1^\infty r^{-2L^2+d-1}~dr,
\\
&\le \lrpar{\frac{C_3}{\sqrt{d}}}^d \cdot \lrpar{\frac{1}{2L^2-d}} \le \lrpar{\frac{C_3}{\sqrt{d}}}^d,
\end{align*}
where on the last line we used the assumption that $2L^2 \ge d+2$ and the standard estimate for the surface area of the unit sphere. This gives us an estimate for the right hand side of \eqref{eq:esseen-reduce-3} and we deduce that 

\[
\mathcal{L}(U^\top X,t\sqrt{d}) \le \frac{1}{\det (U^\top U)^{1/2}}\lrpar{\frac{C_1tL}{u}}^d\lrpar{\lrpar{\frac{C_2}{\sqrt{d}}}^d+\lrpar{\frac{C_3}{\sqrt{d}}}^d} \le \lrpar{\frac{CL}{\sqrt{d}u}}^d\frac{t^d}{\det(U^\top U)^{1/2} }. 
\]
Since, for $t < t_0$ we can replace $t$ with $\max(t,t_0)$ we conclude as desired. 
\end{proof}    
The next step is to show that the upper bound on $\cL\lrpar{UX,t\sqrt{d}}$ can be derived from $\logRLCD{X}{L}{u}{H^\perp}$ instead of $\logRLCD{X}{L}{u}{U}$. We do this by showing that
The Randomized Logarithmic LCD of a subspace and of its projection matrix are the same.
\begin{proposition}\label{prop:matrix-vs-subspace}
Let $L > 0, u \in (0,1)$ and let $X \in \R^n$ be a random vector. Let $U \in \R^{d \times n}$ satisfy $UU^\top = I_d$ and let $E = {\rm Im}(U^\top)$. Then 

\begin{equation}\label{eq:RLCD:subspace+matrix}
\logRLCD{X}{L}{u}{E} = \logRLCD{X}{L}{u}{U}.
\end{equation}
\end{proposition}
\begin{proof}
This is immediate from the observation that 
for every $v \in E$ there exists $x \in E \cap \bS^{n-1}, m \in \R$ and $y \in \R^n$ such that $v = mx = U^\top y$ with $\norm{v}_2 = m = \norm{y}_2$.
\end{proof}
To conclude, we simplify need to show that the upper bound obtained for $\cL\lrpar{UX,t\sqrt{d}}$ can be used to upper bound $\cL\lrpar{P_{H^\perp}X,t\sqrt{d}}.$ 
\begin{corollary}\label{cor:sbp-for-projection}
Let $X \in \R^n$ be a random vector. Let $P_E$ denote the orthogonal projection matrix for $E$. Denote $D(E) := \logRLCD{X}{L}{u}{E}$ Write $t_0 := \sqrt{d}/D$. Let $L \ge 1, u \in (0,1)$. If $2L^2 \ge d+2$ then
\begin{equation}\label{eq:sbp:matrix:projection}
\Levy(P_E \xi, t\sqrt{d}) \leq \lrpar{\frac{CL}{u}}^d
\lrpar{\frac{\max\left\{t,t_0\right\}}{\sqrt{d}}}^d.
\end{equation}
\end{corollary}
\begin{proof}
Since $P_E$ is an orthogonal projection matrix there exists $U \in \R^{d \times n}$ satisfying $UU^\top = I_d, U^\top U = P_E$. Since $U^\top$ is an isometry from $\R^d$ to ${\rm Im}(U^\top)$ it preserves the small ball probability of any random vector in $\R^d$. Therefore 
\[
\Levy(P_E X,t\sqrt{d}) = \Levy(U^\top U \xi,t\sqrt{d}) = \Levy(UX, t\sqrt{d}).
\]
By \eqref{eq:RLCD:subspace+matrix} of \refP{prop:matrix-vs-subspace} we have that $\logRLCD{X}{L}{u}{U} = \logRLCD{X}{L}{u}{E}$.
The result then follows from \refP{prop:sbp:matrix:general}.
\end{proof}
\section{Discretization}\label{Sec:Disc}
Recall that to prove the distance theorem our second objective is to show that the Randomized Logarithmic LCD of $H^\perp$ is large. This is typically done by using a special type of $\eps$-net over the level sets, in terms of LCD, of $H^\perp$. This net should approximate $H^\perp$ in terms of distance and be sufficiently small. It can then be shown that the existence of such a net implies that the LCD of $H^\perp$ is large. In \cite{RV} the $\epsilon$ net only approximated distance in terms of operator norm (i.e. $\norm{A(x-y)}_2 \leq \norm{A}\norm{x-y}_2$). This sufficed since $A$ being made up of independent sub-gaussian entries meant that $\norm{A}$ was, with exponentially small failure probability, of order $\sqrt{n}$. In this case for every $x \in H^\perp$ and its closest net point $y$ the quantity $\norm{A(x-y)}_2$ is of order $\eps\sqrt{n}$, which is sufficient to achieve the implication. In the case where the entries of $A$ are heavy-tailed the operator norm can be prohibitively large, making net-approximations in terms of the operator norm inapplicable. In the work of Livshyts \cite{GL}, a `lattice-like' net was introduced that could control $\norm{A(x-y)}_2$ in terms of the regularized Hilbert-Schmidt norm (see \eqref{eq:regularized-HS} and \refL{lem:net-approx}). Since regularized Hilbert-Schmidt norm is typically of the same order as Hilbert-Schmidt norm (see \refL{lem:concentration-of-HS}), the Hilbert-Schmidt norm assumption makes it again possible to assume that $\norm{A(x-y)}_2$ is of order $\eps\sqrt{n}$. This same type net was again used in \cite{LTV}.We will essentially be using the same net as in \cite{LTV} but with properties now being proven in terms of Randomized Logarithmic LCD. We now recall the relevant discretization results from \cite{LTV} that we will use.
\begin{definition}[Definition 3.1 \cite{LTV}]\label{def:lattice} Given a weight vector $\alpha \in \R^n$ and a resolution parameter $\eps \in (0,1]$ we consider the set of approximately unit vectors whose entries are quantized at scale $\alpha_i\eps/\sqrt{n}$. In particular 
\begin{equation}\label{eq:lattice}
\Lambda_\alpha(\eps) := \lrpar{\frac{3}{2}B_2^n\setminus \frac{1}{2}B_2^n} \cap \lrpar{\frac{\alpha_1 \eps}{\sqrt{n}}\Z \times \cdots \times \frac{\alpha_1 \eps}{\sqrt{n}}\Z}.
\end{equation}
 \end{definition}
Next, given $\kappa > e$ we define 
\[
\Omega_\kappa := \{\alpha \in [0,1]^n~:~\prod_{i = 1}^n \alpha_i \geq \kappa^{-n}\}.
\]
\begin{lemma}[Lemma 3.4 \cite{LTV}]\label{lem:net-of-net}
For any $\kappa > e$ there exists an absolute constant $C > 0$ and a set $\cF \subset \Omega_{e\kappa}$ of size $(C \log \kappa)^n$ such that every element $\beta$ in $\Omega_\kappa$ dominates some element $\alpha$ in $\cF$ ($\alpha_i \le \beta_i$ for all $i$).
\end{lemma}
\begin{definition}[Definition 3.5 \cite{LTV}]\label{def:net}
Let $\cF \subset \R^n$ be the set guaranteed from \refL{lem:net-of-net}. We define
\begin{equation}\label{eq:net}
\Lambda^\kappa(\eps) := \bigcup_{\alpha \in \cF} \Lambda_\alpha(\eps).
\end{equation}
\end{definition}
\begin{remark}[Remark 3.6 \cite{LTV}] \label{rem:net-size}
For every $\kappa > e$ there exists $C_\kappa > 0$ depending on $\kappa$ such that $|\Lambda^\kappa(\eps)| \leq (C_\kappa/\eps)^n$.
\end{remark}
\begin{definition}\cite{GL}\label{def:regularized-HS}
Given  $\kappa > e$ and $A \in \R^{N \times n}$ the regularized Hilbert-Schmidt norm is defined as
\begin{equation}\label{eq:regularized-HS}
\cB_\kappa(A) = \min\left\{\sum_{i = 1}^n \alpha_i^2\norm{A_i}_2^2~:~\alpha \in \Omega_\kappa\right\}.
\end{equation}
\end{definition}

\begin{lemma}[Theorem 3.7 \cite{LTV}]\label{lem:net-approx}
Given $A \in \R^{N \times n}$ and $\eps \in (0,1)$, every unit vector $x$ can be approximated by some element $y \in \Lambda^\kappa(\eps)$ such that 
\begin{equation}\label{eq:rounding-guarantee:2}
\norm{x-y}_\infty \leq \frac{\eps}{\sqrt{n}}, \norm{A(x-y)}_2 \leq \frac{\eps\sqrt{\cB_\kappa(A)}}{\sqrt{n}}.
\end{equation}
\end{lemma}
\begin{lemma}[Lemma 3.11 \cite{GL}]\label{lem:concentration-of-HS}
Let $A$ be a random matrix with independent columns. Then for any $\kappa > e$ we have 
\[
\bP\lrpar{ \cB_\kappa(A) \geq 2\E \norm{A}_{HS}^2} \leq \lrpar{\frac{\kappa}{\sqrt{2}}}^{-2n}.
\]
\end{lemma}
So far the results stated show how the net may be a good approximation of $H^\perp$ in terms of distance. We now address how the net size scales with the level set of $H^\perp$ it approximates. Given a random matrix $A \in \R^{N \times n}$, and parameters $L,D > 0, u \in (0,1)$ we write
\begin{equation}\label{eq:levelsets}
\begin{split}
S^A_{L,u,D} &:= \left\{x \in \frac{3}{2}B_2^n \setminus \frac{1}{2}B_2^n~:~\logRLCD{A}{L}{u}{x} \in [D,2D] \right\}, \\
\tilde{S}^A_{L,u,D} &:= \left\{x \in \frac{3}{2}B_2^n \setminus \frac{1}{2}B_2^n~:~\logRLCD{A}{2L}{6u}{x} \leq 2D, \logRLCD{A}{L/2}{u/6}{x} \geq D \right\},
\end{split}
\end{equation}
The definitions of $S^A_{L,u,D}$ and $\tilde{S}^A_{L,u,D}$ are chosen so as to coincide with \refL{lem:stability} with $r_1 = 1/2,r_2 = 3/2$. One should imagine that a point $v \in H^\perp$ in the level set $S^A_{L,u,D}$ is being approximated by some net point contained in $\tilde{S}^A_{L,u,D}$.
\begin{lemma}\label{lem:existence-of-net}
Fix any $\eps \in (0,1/2), \kappa > e, L > 0, u  \in (0,1)$. Let
$A \in \R^{N \times n}$ be a random matrix with independent columns, and whose rows $A^i$ satisfy

\begin{equation}\label{eq:existence-of-net}
\eps^2 \Var(A^i) \leq \frac{n}{8}\frac{L^2}{D^2} \cdot \log_+\lrpar{\frac{Du}{L}}.
\end{equation}

Then with probability at least $1 - (\kappa/\sqrt{2})^{-2n}$ every unit vector $x$ in $S^A_{L,u,D}$ is approximated by some vector $y$ in $\Lambda^\kappa(\eps) \cap \tilde{S}^A_{L,u,D}$ such that 
\begin{equation}\label{eq:rounding-guarantee:3}
\norm{x-y}_\infty \leq \frac{\eps}{\sqrt{n}},~~\norm{A(x-y)}_2 \leq \frac{\eps \sqrt{2\E\norm{A}_{\text{HS}}^2}}{\sqrt{n}}
\end{equation}
\end{lemma}
\begin{proof}
We follow the proof strategy from \cite{LTV}. By \refL{lem:concentration-of-HS} the event $\{\cB_{\kappa}(A) < 2\E\norm{A}_{HS}^2\}$
occurs with probability at least $1 - \lrpar{\frac{\kappa}{\sqrt{2}}}^{-2n}$. Conditioned on this event, by \refL{lem:net-approx} for every $x \in \bS^{n-1} \cap S_{L,u,D}^A$ there exists $y \in \Lambda^\kappa(\eps)$ satisfying \eqref{eq:rounding-guarantee:2} with $\cB_\kappa(A) \leq 2\E\norm{A}_{\text{HS}}^2$. Since the rows of $A$ satisfy \eqref{eq:existence-of-net} 
and $\norm{y}_2 \in [1/2,3/2]$ we may apply \refL{lem:stability} with $r = \eps/\sqrt{n}, r_1 = 3/2, r_2 = 1/2$ to get 
\[
\logRLCD{A}{2L}{6u}{y} \leq \logRLCD{A}{L}{u}{x}\leq \logRLCD{A}{L/2}{u/6}{y}.
\]
Since $\logRLCD{A}{L}{u}{x} \in [D,2D]$ we conclude that $y \in \tilde{S}^A_{L,u,D}$.
\end{proof}
\refL{lem:existence-of-net} implies that an $\eps$-net made from sets of the form $\Lambda^\kappa(\eps)$ can approximate level sets of $H^\perp$ in terms of distance and LCD. We now want a way to control the cardinality of the $\eps$-net. To that end consider the following set of lattice points:
\begin{equation}\label{eq:grid}
\Lambda := \lrpar{\frac{3}{2}B_2^n \cap \{ x \in \R^n~:~\#\left\{i~:~|x_i| \geq \frac{\rho}{2\sqrt{n}}\right\} \geq \delta n\}} \cap \lrpar{\frac{\lambda_1}{\sqrt{n}}\Z \times \cdots \times \frac{\lambda_n}{\sqrt{n}}\Z}.
\end{equation}
In \cite{LTV} the following theorem was proven regarding the randomized LCD of a typical lattice point of \eqref{eq:grid}.
\begin{theorem}[Theorem 4.1 \cite{LTV}]\label{thm:unstructured} There exist $n_0 = n_0(b,k)$ and $u = u(b,K) \in (0,1/4)$ such that the following holds. Let $n \ge n_0$ and $X \in \R^n$ be a random vector with uniformly anti-concentrated entries that satisfies
\[
\E\norm{X}_2^2 \leq \frac{1}{8}(1-b)\delta \gamma^2 n^2.
\]
 Fix numbers $\lambda_1,\cdots \lambda_n$ satisfying $6^{-n} \leq \lambda_i \leq 0.01$ and let $W$ be a vector uniformly distributed on the set $\Lambda$ defined in \eqref{eq:grid}. Then 
 \[
 \bP_W\left\{\RLCD{X}{\gamma\sqrt{n}}{u}{W} < \min_i 1/\lambda_i\right\} \leq (C\gamma)^{cn},
 \]
 where $C,c > 0$ are constants depending only on $b$ and $K$. In particular given $U \geq 1$ there exists $\gamma = \gamma(b,K)$ such that $(C\gamma)^{cn}$ may be replaced with $U^{-n}$.
\end{theorem}
Combining \refL{lem:RLCD-vs-logRLCD} and \refT{thm:unstructured} we obtain a version of \refT{thm:unstructured}  in terms of Randomized Logarithmic LCD.
\begin{corollary}\label{cor:unstructured-log}
    Let $L,\gamma > 0$.Then there exist $n_0 = n_0(b,k)$ and $u = u(b,K) \in (0,1/4)$ such that the following holds: Let $X \in \R^n$, with $n \ge n_0$, be a random vector with uniformly anti-concentrated satisfying
\[
\E|X|^2 \leq \frac{1}{8}(1-b)\delta \gamma^2 n^2.
\]
 Fix numbers $\lambda_1,\cdots \lambda_n$ satisfying $6^{-n} \leq \lambda_i \leq 0.01$ and let $W$ be a vector uniformly distributed on the set $\Lambda$ defined in \eqref{eq:grid}. Then 
 \[
 \bP_W\left\{\logRLCD{X}{L}{u}{W} < \min\lrpar{\frac{2L}{u}e^{n(\gamma/L)^2},\min_i 1/\lambda_i}\right\} \leq (C\gamma)^{cn},
 \]
  where $C,c > 0$ are constants depending only on $b$ and $K$. In particular given $U \geq 1$ there exists $\gamma = \gamma(b,K)$ such that $(C\gamma)^{cn}$ may be replaced with $U^{-n}$.
 \end{corollary}
 Recall from \refP{prop:incomp} that  $\bS^{n-1} \cap H^\perp$ likely contains no compressible vectors. Therefore using \refL{lem:existence-of-net} and \refC{cor:unstructured-log} we can prove the existence of the desired $\eps$-net over a given level set $S^A_{L,u,D} \cap (\bS^{n-1} \cap H^\perp)$.
 \begin{lemma}\label{lem:distance:1}
Given $U \geq 1$ there exist $n_0 =
n_0(b,K), u = u(b,K), \gamma = \gamma(b,k,U)$
such that the following holds: Let $L \geq 1, 0 < D < \overline{D} \leq (2L/u)e^{n(\gamma/L)^2}$ and $0 < \eps \leq 1/\overline{D}$. Let $n \geq n_0$
 and let $A \in \R^{N \times n}$ be a random matrix that satisfies the Hilbert-Schmidt norm condition, has entries that are uniformly anti-concentrated and has rows $A^i$ that satisfy
\begin{equation}\label{eq:distance:1}
\E|A^i|^2 \leq \min\left\{\frac{1}{8}(1-b)\delta\gamma^2n^2, \frac{nL^2}{8}\log_+\lrpar{\frac{\overline{D}u}{L}}\right\}.
\end{equation}
Then there exists constants a non-random set $\cN \subseteq \tilde{S}^A_{L,u,D}$ of cardinality at most $(U\eps)^{-n}$
having the following
properties:
\begin{enumerate}
\item With failure probability at most $e^{-n}$, every point $x$ in $S^A_{L,u,D} \cap\incomp(\delta, \rho)$ is approximated by some point $y$ in $\cN $ such that $\norm{x-y}_\infty \leq \eps/\sqrt{n}$ and $\norm{A(x-y)}_2 \leq \eps\sqrt{N}$. 
\item 
If $x \in S^A_{L,u,D} \cap\incomp(\delta, \rho)$ and $\norm{x-y}_\infty \leq \eps/\sqrt{n}$ then
\[
\logRLCD{A^i}{L}{u}{x} \geq \overline{D} \implies \logRLCD{A^i}{L/2}{u/6}{y} \geq \overline{D} 
\]
for all rows $A^i$.
\end{enumerate}
\end{lemma}
The proof is analogous to the proof of Proposition 5.1 in \cite{LTV}.
\begin{proof}
We define $\eps' := \min(\eps/(2K),\rho/2)$. $n_0$ and $u$ are chosen so as to satisfy \refC{cor:unstructured-log}. We take $\cN \subset \Lambda^3(\eps') \cap \tilde{S}_{L,u,D}^A$ to be the subset of points whose $\ell_\infty$ distance to at least 1 point in $S_{L,u,D}^A \cap\incomp(\delta,\rho)$ is at most $\eps'/\sqrt{n}$. Note that, with failure probability at most $(3/\sqrt{2})^{-2n} \le e^{-n}$, every unit vector in $S^{A}_{L,u,D}$ satisfies the first condition due to \refL{lem:existence-of-net} and the Hilbert-Schmidt norm assumption. Because the rows of $A$ satisfy \eqref{eq:distance:1} we may apply \refL{lem:stability} and conclude \eqref{eq:stability:guarantee:1} for each row of $A$. We are left with bounding $|\cN|$. From \eqref{eq:net} and \eqref{eq:lattice} we know that $\Lambda^3(\eps')$ is a union of sets of the form $\Lambda_\alpha(\eps')$. Since a compressible vector has at least $\delta n$ entries of size at least $\rho/\sqrt{n}$ and $\rho/\sqrt{n} - \eps'/\sqrt{n} \ge \rho/(2\sqrt{n})$ we have that $\cN \cap \Lambda_\alpha(\eps')$ is a subset of a set of the form \eqref{eq:grid}. Since $\min_i \alpha_i \eps' \le 1/\overline{D} \le (2L/u)e^{n(\gamma/L)^2}$ and the rows satisfy \eqref{eq:distance:1} we may apply \refC{cor:unstructured-log} to each of these subsets to bound their size by $|\Lambda_\alpha(\eps')|(C\gamma)^{cn}$. Therefore the size of $\cN$ is at most $|\Lambda^3(\eps')|(C\gamma)^{cn} \le (C/\eps)^{n}(C \gamma)^{cn}$. By making $\gamma$ sufficiently small (and depending only on $U,b$ and $k$), the size is at most $(U\eps)^{-n}$.
\end{proof}
 \section{The distance theorem}\label{Sec:Distance:Theorem}
 \subsection{$\logRLCD{X}{L}{u}{H^\perp}$ is large}
With \refL{lem:distance:1} in hand we are ready to prove that $\logRLCD{X}{L}{u}{H^\perp}$ is large (see \refT{thm:distance:3}). Much of this section is analogous to Section 5 of \cite{LTV}. To begin we recall a `tensorization' lemma.
\begin{lemma}[Lemma 2.2 \cite{RV-square}]\label{lem:tensorization}
Let $\xi_1,\cdots,\xi_n$ be a collection of independent random variables satisfying $\bP(|\xi_i| \le \eps) \le K\eps$ for all $\eps > \eps_0 > 0$.
Then 
\[
\bP\lrpar{\sum_{i = 1}^d |\xi_i|^2 \leq t^2d} \leq (CK\eps)^n
\]
for all $\eps > \eps_0$, for some $C > 0$.
\end{lemma}
We now discuss the proof strategy for \refT{thm:distance:3}. In order to show that $\logRLCD{X}{L}{u}{H^\perp}$ is large we would first like to show that $\logRLCD{A}{L}{u}{H^\perp}$ is large and then use this information to show that $\logRLCD{X}{L}{u}{H^\perp}$ is large. This will follow from a proof by contradiction. Assuming $\logRLCD{A}{L}{u}{H^\perp}$ is large, if $\logRLCD{X}{L}{u}{H^\perp}$ is small then there must exist $x \in H^\perp$ for which $\logRLCD{A|X}{L}{u}{x}$ is small (say in some level set $[D,2D]$). Here $A|X$ denotes the matrix obtained by appending $X$ to $A$ as its top row. Using the net from \refD{lem:distance:1} and the fact that $A^\top x = 0$ it will then follow that there exists a net point $y \in \Lambda$ such that $\norm{A^\top y}_2 \leq \norm{(A|X)^\top (x-y)}_2$, where the latter quantity is known to be not too large. On the other hand because $\logRLCD{A}{L}{u}{H^\perp}$ is large the properties of $\Lambda$ and \refP{prop:sbp:matrix:general} will imply a lower bound on $\norm{A_i^\top y}_2$ for every column $A_i$ of $A$. \refL{lem:tensorization} will then imply that $\norm{A^\top y}_2$ is likely large, giving a contradiction. The proof that $\logRLCD{A}{L}{u}{H^\perp}$ is large will follow in a similar way using a recursive argument on the rows of $A^\top$.
\par
There are, however, a few problems with this particular argument. Firstly, in order to guarantee $\norm{A_i y}_2$ is likely not too small it is necessary to use \refP{prop:sbp:matrix:general} with $L = 1$. Since $L$ will be chosen to be of order $\sqrt{d}$ we will have to appeal to \refC{cor:RLCD:compare}, which requires a sufficiently large lower bound on $\logRLCD{A_i}{L}{u}{y}$. The lower bound obtainable through \refL{lem:RLCD:lb} would then require that the second moment of $A_i$ be of order at most $n^2/d$, which does not necessarily hold for all columns of $A$. Therefore we will have to restrict our attention to columns of $A$ with sufficiently small second moment. The corresponding matrix $Q$ will thus need enough rows so as to guarantee that $\norm{Qy}_2$ is likely large. Luckily, because we assume that $A$ satisfies the Hilbert-Schmidt norm assumption, this will turn out to be so. The second issue is more subtle. Since $X$ only satisfies the second moment assumption, \eqref{eq:RLCD-lb-2} of \refL{lem:RLCD:lb} will only give a non-trivial lower bound on $\logRLCD{X}{\tilde{L}}{u}{x}$, where $\tilde{L}$ belongs to $\cN$ from \eqref{eq:RLCD-lb-3} and depends on $x$. In particular the contradiction argument will not directly show that $\logRLCD{X}{L}{u}{H^\perp}$ is large. Instead, we will 
 partition $H^\perp$ as $H^\perp = \cup_{\ell \in \cN} H^\perp_\ell$, with $H^\perp_\ell$ consisting of all vectors with a non-trivial lower bound for parameter $\ell$. We will then show that $\logRLCD{X}{\ell}{u}{H^\perp_\ell}$ is large for every $\ell \in \cN$ and then use the fact that $\cN \subseteq [L,2L]$ to deduce that $\logRLCD{X}{L}{u}{H^\perp}$ is large.
\begin{proposition}\label{prop:distance:2}
There exists $n_0 = n_0(b,K), u = u(b,K), \gamma = \gamma(b,K)$ such that the following holds:
Pick $L \ge 2$. Let $d \le n/\log (\overline{D})$ and $0 < D \leq \overline{D} \leq \frac{2L}{u}e^{(\gamma/L)^2n}$. Let $A = \R^{(n-d+1) \times n}$ be a random matrix with uniformly anti-concentrated entries that satisfies the Hilbert-Schmidt norm assumption and whose rows satisfy
\begin{equation}\label{eq:distance:2:1}
\E|A^i|^2 \leq 
\begin{cases}
    \min \{(1-b)\delta\gamma^2n^2/8, \frac{nL^2}{8}\frac{\overline{D}^2} {D^2}\log_+\lrpar{\frac{Du}{L}},\frac{nL^2}{8}\log_+\lrpar{\frac{\overline{D}u}{L}}\} 
    & \text{ if } i > 1, 
    \\
    \min \{(1-b)\delta \gamma^2n^2/8,\frac{nL^2}{8}\log_+\lrpar{\frac{\overline{D}u}{L}} \}
    & \text{ if } i = 1.
\end{cases}
\end{equation}

Furthermore suppose that
\begin{equation}\label{eq:distance:2:2}
\min_{x \in \incomp(\delta,\rho)}\logRLCD{A^{(1)}}{2}{u}{x} \geq 2L/u.
\end{equation}
Then
\begin{equation}\label{eq:distance:2:3}
\bP[\exists x \in\incomp(\delta,\rho) \cap S_{L,u,D}^A ~s.t.~ Ax = 0, \logRLCD{A^{(1)}}{L}{u}{x} \geq \overline{D}] \leq e^{-cn},
\end{equation}
where $c > 0$ is an absolute constant.
\end{proposition}
\begin{proof}
Take $n_0,u,U,\gamma$ from \refL{lem:distance:1} with $U$ and $\gamma$ to be determined later. Since the rows satisfy \eqref{eq:distance:2:1} we may take $\cN$ to be set guaranteed by \refL{lem:distance:1} with $\eps := 1/\overline{D}$. Since every point in $\cN$ is within $\ell_\infty$ distance $\eps$ of some incompressible vector, the rows satisfy \eqref{eq:distance:2:1} and $A$ satisfies \eqref{eq:distance:2:2}, by \refL{lem:stability} we have 
\[
\inf_{y \in \cN}\logRLCD{A^i}{1}{u/6}{y} \geq 2L/u.
\]
Therefore by \refL{lem:RLCD:compare} $\logRLCD{A^i}{1}{u/6}{y} \leq \logRLCD{A^i}{1}{u/6}{x}$. 
Now take $\cN' \subset \cN$ to be the subset whose points $z$ satisfy  $\logRLCD{A^{(1)}}{1}{u/6}{z} \ge \overline{D}$.  Applying \refP{prop:sbp:matrix:general} with $U := z$ gives
\[
\Pr\left[|A^iz| \leq \frac{1}{\overline{D}} \right] \lsim \overline{D}^{-1}.
\]
Applying \refL{lem:tensorization} to $\sum_{i > 1}|A_iz|^2$ gives
\[
\Pr\left[\norm{A^{(1)}z}_2 \leq \frac{\sqrt{n-d}}{\overline{D}} \right] \le (c\overline{D})^{d-n}.
\]
Note now that if the event
\[
\{\exists x \in\incomp(\delta,\rho) \cap S_{L,u,D}^A ~s.t.~ Ax = 0, \logRLCD{A^{(1)}}{L}{u}{x} \geq \overline{D}\}
\]
occurs then, by \refL{lem:stability} and \refL{lem:distance:1} some point $z \in \cN'$ satisfies $\norm{A^{(1)}z}_2 \le \sqrt{n-d}/\overline{D}$ with probability $e^{-n}$. Taking a union bound over $\cN'$ we have  
\[
\bP(\exists x \in\incomp(\delta,\rho) \cap S_{L,u,D}^A ~s.t.~ Ax = 0, \logRLCD{A^{(1)}}{L}{u}{x} \geq \overline{D}) \le e^{-n} + (U/\overline{D})^{-n}(c\overline{D})^{d-n}  \le e^{-n} + (e/(cU))^{n}.
\]
The desired probability is then achieved by making $\gamma$ sufficiently small, so as to make $U$ sufficiently large.
\end{proof}
Note that \refL{lem:distance:1} bounds the probability that the Randomized Logarithmic LCD of a vector lies in the interval $[D,2D]$. By taking a suitable dyadic decomposition of a range $\left[\underline{D},\overline{D}\right] $ and applying \refP{prop:distance:2} to each interval, the result can be extended to a range.
\begin{corollary}\label{cor:distance:2}
There exists $n_0 = n_0(b,K), u = u(b,K), \gamma = \gamma(b,K)$ such that the following holds:
Pick $L \ge 2$. Let $d \le n/\log (\overline{D})$ and $0 < \underline{D} \le D \leq \overline{D} \leq \frac{L}{u}e^{(\gamma/L)^2n}$. Let $A = \R^{(n-d+1) \times n}$ be a random matrix with uniformly anti-concentrated entries that satisfies the Hilbert-Schmidt norm assumption, \eqref{eq:distance:2:1} and whose rows satisfy 
\begin{equation}\label{eq:distance:2:4}
\E|A^i|^2 \leq 
\begin{cases}
    \min\{(1-b)\delta\gamma^2n^2/8, \frac{nL^2}{8}\frac{\overline{D}^2} {\underline{D}^2}\log_+\lrpar{\frac{\underline{D}u}{L}},\frac{nL^2}{8}\log_+\lrpar{\frac{\overline{D}u}{L}} \} & \text{ if }i > 1 \\
    \min \{(1-b)\delta \gamma^2n^2/8,\frac{nL^2}{8}\log_+\lrpar{\frac{\overline{D}u}{L}} \}& \text{ if } i = 1.
\end{cases}
\end{equation}
Then
\begin{equation}\label{eq:distance:3}
\bP[\exists x \in\incomp(\delta,\rho) ~s.t.~ Ax = 0, \logRLCD{A}{L}{u}{x} \in [\underline{D},\overline{D}],\logRLCD{A^{(1)}}{L}{u}{x} \geq \overline{D}] \leq e^{-cn},
\end{equation}
where $c > 0$ is an absolute constant.
\end{corollary}
%ur an impressive man -Hannah
%thanks -Manuel (like one year later lol)
\begin{proof}
We may assume that $\overline{D}/\underline{D}$ is a power of 2 by replacing $\overline{D}$ with a number that is only twice as large, being at most $(2L/u)e^{(\gamma/L)^2n}$. Take $n_0,u,\gamma$ from \refP{prop:distance:2}. Note that for every $p > 0$ the function $x \mapsto \frac{1}{x^2}\log(px)$ is unimodal on the interval $(1/p,\infty)$. In particular over any interval in $(1/p,\infty)$ the function is minimized at one of the endpoints. Hence
\[
\min_{D \in \left[\underline{D},\overline{D}\right]} \frac{nL^2}{8}\frac{\overline{D}^2}{D^2}\log_+\lrpar{\frac{Du}{L}} = \min\left\{  \frac{nL^2}{8}\log_+\lrpar{\frac{\overline{D}u}{L}}, \frac{nL^2}{8}\frac{\overline{D}^2}{\underline{D}^2}\log_+\lrpar{\frac{\underline{D}u}{L}} \right\}.
\]
Next note that the set of incompressible vectors $x$ satisfying $\logRLCD{A}{L}{u}{x} \in [\underline{D},\overline{D}]$ are contained in $\cup_{2^{-i}\overline{D} \ge \underline{D}} S^A_{L,u,2^{-i}\overline{D}}$.
By \refP{prop:distance:2} 
\[
\Pr[\exists x \in\incomp(\delta,\rho) \cap S_{L,u,2^{-i}\overline{D}}^A ~s.t.~ Ax = 0, \logRLCD{A^{(1)}}{L}{u}{x} \geq \overline{D}] \leq e^{-n},
\]
Since $\log_2(|\overline{D}-\underline{D}|) \lsim n$ we may take a union bound over all choices for $2^{-i}\overline{D}$ and conclude that  
\[
\bP[\exists x \in\incomp(\delta,\rho) ~s.t.~ \logRLCD{A}{L}{u}{x} \in \left[\underline{D},\overline{D}\right], Ax = 0, \logRLCD{A^{(1)}}{L}{u}{x} \geq \overline{D}] \leq Cne^{-n} \le e^{-cn}.
\]
\end{proof}

\begin{theorem}\label{thm:distance:3}
There exists $n_0 = n_0(b,K), u = u(b,K), \lambda = \lambda(b,K)$ such that the following holds: Let $n \geq n_0$ and $1 \leq d \leq \lambda n/\log n$.
Let $A \in \R^{n \times (n-d)}$ be a random matrix that satisfies the Hilbert-Schmidt norm assumption and has uniformly anti-concentrated entries. Let $X \in \R^n$ be a random vector that satisfies the second moment assumption and has uniformly anti-concentrated entries. Then
\begin{equation}\label{eq:distance:3}
\bP[\exists x \in\incomp(\delta,\rho) \text{ and }A^\top x = 0~s.t.~\logRLCD{X}{L}{u}{x} \leq c_1\sqrt{d}e^{c_2n/d}] \leq e^{-c_3 n},
\end{equation}
where $L \in [2\sqrt{d},C\sqrt{d}]$ and $C,c_1,c_2,c_3 > 0$ depend only on $b$ and $K$.
\end{theorem}

\begin{proof}
We pick $u,n_0$ and $\gamma$ as guaranteed by \refC{cor:distance:2}. Since $L \le C\sqrt{d} \ll n$ and $X$ satisfies the second moment condition by \refL{lem:RLCD:lb} we can write $\incomp(\delta,\rho) \cap \operatorname{null}(A) = \cup_{\ell \in \cN} \cS_\ell$ where $\cS_\ell$ consists solely of points $x$ that satisfy 
\begin{equation}\label{eq:distance:3:proof:1}
\logRLCD{X}{\ell}{u}{x} \ge \frac{\ell}{u} + C_1,
\end{equation}
where $C_1 > 0$ is a sufficiently large constant. Suppose we prove that
\[
\inf_{\ell \in \cN} \bP(\exists x \in \cS_\ell ~s.t~ \logRLCD{X}{\ell}{u}{x} \le c_1\sqrt{d}e^{c_2n/d}) \le e^{-c_3n}.
\]

Since $\ell^2 \log_+(\theta u/\ell) \ge L^2\log_+(\theta(u/2)/L)$ for all $\theta \ge c_1\sqrt{d}e^{c_2n/d}$  from the definition of Randomized Logarithmic LCD we may conclude that 
\[
\inf_{\ell \in \cN} \bP(\exists x \in \cS_\ell ~s.t~ \logRLCD{X}{L}{u/2}{x} \le c_1\sqrt{d}e^{c_2n/d}) \le e^{-c_3n}.
\]
Since $|\cN| \lsim L \ll n$ we may apply a union bound over all $\cS_\ell$ and conclude \eqref{eq:distance:3}.
Thus, we are left with bounding 
\[
\bP(\exists x \in \cS_\ell ~s.t~ \logRLCD{X}{L}{u/2}{x} \le c_1\sqrt{d}e^{c_2n/d}),
\]
where $\ell \in \cN$ is arbitrary.

For convenience, given a matrix $B$ and a row vector $Y$ we write $B|Y$ to denote the matrix obtained by appending $Y$ to $B$ as its first row. We now introduce some parameters. 
\begin{align}
    \label{val:LB}
    \underline{D} &= \ell/u + C_1, \\
    \label{val:UB}
    \overline{D} &:= \max\{t \le (\ell/u)e^{(\gamma/\ell)^2n}~:~t/(C_2\ell/u) \text{ is a power of }2\},  \\
    \label{val:D} 
    D_i &:= \overline{D}2^{-i}, \\
    \label{val:tau}
    \tau &:= \arg(D_i = C_2\ell/u),
    \\
    \label{val:q_i}
    q_i &:= \min\lrpar{\frac{c_2n^2}{\ell^2}, \frac{n\ell^2}{8}\log_+\lrpar{\frac{D_iu}{\ell}}}, \\
    \label{val:Q_i}
    Q_i &:= \text{sub matrix of }A^\top \text{ consisting of all rows with second moment at most }q_i. 
\end{align}
Assuming $L \in [2\sqrt{d},C\sqrt{d}]$, $\overline{D}$ is of the form the LCD is required to satisfy in \eqref{eq:distance:3:proof:1}.
We choose $c_2$ to be a sufficiently small constant such that any vector $X$ with second moment at most $c_2n^2/\ell^2$ satisfying the assumptions of \refL{lem:RLCD:lb} for \eqref{eq:RLCD-lb-1} will satisfy $\inf_{x \in \incomp(\delta,\rho)} \logRLCD{X}{t}{u}{x} \ge C_2t/u$, for any value of $t$, where $C_2$ is a sufficiently large constant. We also note that 
\[
\frac{n\ell^2}{8}\log_+\lrpar{\frac{\overline{D}u}{\ell}} \le \frac{n\ell^2}{8}\log_+(\exp((\gamma/\ell)^2n)) \le \gamma^2n^2/8.
\]
We now define a number of events:
\begin{align}
        \cH_i &:= \{\exists x \in \cS_\ell ~s.t.~ \logRLCD{Q_i}{\ell}{u}{x} \in [C_2\ell/u,D_i)\}, \\
        \cE_i &:= \{\exists x \in \cS_\ell~s.t.~\logRLCD{Q_{i+1}}{\ell}{u}{x} \in [C_2\ell/u,D_i) \text{ and } \logRLCD{Q_{i+1}}{\ell}{u}{x} \geq D_{i+1}\}, \\
        \cF_i &:= \bigcup_{Y \in Q_i\setminus Q_{i+1}}\{ \exists x \in \cS_\ell~s.t.~\logRLCD{Q_{i+1}|Y}{L}{u}{x} \in [C_2\ell/u,D_i) \text{ and } \logRLCD{Q_{i+1}}{\ell}{u}{x} \geq D_{i+1}\}.
    \end{align}
We claim that $
    \cH_i \subset \cE_i \cup \cF_i \cup \cH_{i+1}$. To see this we case on whether or not $Q_i = Q_{i+1}$

    \begin{enumerate}
        \item $Q_i = Q_{i+1}$. If $\cH_i \not \subset \cE_i$ then either there exists $x \in \cS_\ell$ such that $\logRLCD{Q_{i+1}}{\ell}{u}{x} < D_{i+1}$ or all such $x$ satisfy $\logRLCD{Q_{i+1}}{\ell}{u}{x} \ge D_i$. Since the last case is disjoint from $\cH_i$ we conclude that $\cH_i \subset \cH_{i+1}$.
        \item $Q_i \neq Q_{i+1}$. If $\cH_i \not \subset \cF_i$ then either there exists $x \in \cS_\ell$ such that $\logRLCD{Q_{i+1}}{\ell}{u}{x} < D_{i+1}$ or all such $x$ satisfy $\logRLCD{Q_{i+1}|Y}{\ell}{u}{x} > D_i$ for any $Y \in Q_i\setminus Q_{i+1}$. In particular this implies a lower bound of the LCD of every row in $Q_i$ so $\logRLCD{Q_i}{\ell}{u}{x} > D_i$. Since the last case is disjoint from $\cH_i$ we conclude that $\cH_i \subset \cH_{i+1}$.
    \end{enumerate}
    By the choice of $c_2$ and definition of $Q_0$, $Q_0$ satisfies $\inf_{x \in \cS_\ell}\logRLCD{Q_0}{\ell}{u}{x} \geq C\ell_2/u$. In particular if $x$ has LCD less that $\overline{D}$ it must be in the range $[C_2\ell/u,\overline{D})$. Therefore 
    \[
    \bP(\exists x \in \cS_\ell~s.t.~\logRLCD{Q}{\ell}{u}{x} < \overline{D}) \le \bP(\cH_0) \le \bP(\cup_{0 \leq i \leq \tau} (\cE_i \cup \cF_i)) \le \sum_{i = 0}^\tau (\Pr[\cE_i] + \Pr[\cF_i]).
    \]
    To bound $\Pr[\cE_i]$ and $\Pr[\cF_i]$ we will apply \refC{cor:distance:2}.
    Suppose the following inequalities hold for all $0 \leq i \leq \tau$.:
\begin{align}
    \label{eq:lcd-Qi} 
    \inf_{x \in \incomp(\delta,\rho)}\logRLCD{Q_0}{2}{u}{x} &\geq 2\ell/u,\\
    \label{eq:size-Qi} 
    \text{the number of rows in }Q_i &\geq n - n/\log(D_i), \\
    \label{eq:var}
    \frac{c_2n^2}{\ell^2} &\leq (1-b)\delta\gamma^2 n^2/8, \\
    \label{eq:Di-vs-D}
    (1/D_i)^2\log_+\lrpar{\frac{D_iu}{\ell}} &\leq (1/\overline{D})^2\log_+\lrpar{\frac{\overline{D}u}{\ell}}.  
\end{align}
    Note that \eqref{eq:Di-vs-D} and \eqref{eq:var} imply 
\[
\min\left\{\frac{1}{8}(1-b)\delta\gamma^2n^2, \frac{n\ell^2}{8}\frac{D_i^2}{\underline{D}^2}\log_+\lrpar{\frac{\underline{D}u}{\ell}},q_i \right\}
= q_i.
\]
\eqref{val:q_i},\eqref{val:Q_i},\eqref{eq:lcd-Qi} and \eqref{eq:size-Qi} then imply, by \refC{cor:distance:2}, that $\bP(\cE_i),\bP(\cF_i) \le e^{-cn}$. Therefore 
\[
\sum_{i = 0}^\tau (\Pr[\cE_i] + \Pr[\cF_i]) \leq Cne^{-cn} \le e^{-cn}.
\]

Our next step is to bound 
$\bP(\exists x \in \cS_\ell~s.t.~\logRLCD{Q|X}{\ell}{u}{x} < \overline{D}, \logRLCD{Q}{L}{u}{x} \ge \overline{D})$.

Suppose the following equation holds:
\begin{equation}
\label{eq:D_-vs-D}
(1/\overline{D})^2\log_+\lrpar{\frac{\overline{D}u}{\ell}} < (1/\underline{D})^2\log_+\lrpar{\frac{\underline{D}u}{\ell}}. 
\end{equation}
Then \eqref{eq:D_-vs-D} and \eqref{eq:var} imply that
\[
\min\left\{\frac{1}{8}(1-b)\delta\gamma^2n^2, \frac{n\ell^2}{8}\frac{\overline{D}^2}{\underline{D}^2}\log_+\lrpar{\frac{\underline{D}u}{\ell}},q_0 \right\}
= q_0.
\]
Since $X$ has the second moment assumption it satisfies $\E\norm{X}_2^2 \le c_3n^2$ where $c_3$ is sufficiently small. In particular we may take $c_3$ to be smaller than $(1-b)\delta\gamma^2/8$, in which case $\E\norm{X}_2^2$ is at most 
$\min\{(n\ell^2/8)\log_+(\overline{D}u/\ell),(1-b)\delta\gamma^2n^2/8\}$.
\eqref{val:Q_i},\eqref{eq:lcd-Qi},\eqref{eq:size-Qi} then imply, by \refC{cor:distance:2}, that the probability is at most $e^{-cn}$. Therefore 
\[
\bP(\exists x \in \cS_\ell~s.t. \logRLCD{X}{L}{u}{x} < \overline{D}) \le e^{-cn} + e^{-cn} \le e^{-cn}.
\]
We are now left with verifying \eqref{eq:lcd-Qi},
\eqref{eq:size-Qi}, \eqref{eq:var}, \eqref{eq:Di-vs-D} and \eqref{eq:D_-vs-D}.
Verifying \eqref{eq:lcd-Qi} is immediate since all rows of $Q_0$ satisfy $\inf_{x \in \incomp(\delta,\rho)}\logRLCD{X}{L}{u}{x} \ge C_2\ell/u$ and we can take $C_2 > 2$. To verify \eqref{eq:var} we use Markov's inequality. Let $n_i$ denote the number of rows in $Q$ with second moment at most $q_i$. Since $A$ satisfies the Hilbert-Schmidt norm assumption, Markov's inequality implies that 
\[
n_i \ge n - d - \frac{Kn(n-d)}{q_i} \ge n - \frac{Kn(n-d)}{(n\ell^2/8)\log_+(D_iu/\ell)} \ge n - \frac{8Kn}{\ell^2\log(D_iu/\ell)}.
\]
For the negative term to be at most $n/\log(D_i)$ it suffices that $8K/\ell^2 \le \log(D_iu/\ell)/\log(D_i)$. Since the map $t \mapsto \log(tu/\ell)/\log(t)$ is increasing for $t \ge \ell/u$ and $D_i \ge C_2\ell/u$ it suffices that $(8K/\ell^2) \le \log(C_2)/\log(C_2\ell/u)$. Since $(\ell^2\log(\ell))^{-1} \to 0 \text{ as } \ell \to \infty$ we can pick $L$ large enough in the range of $[2\sqrt{d},C\sqrt{d}]$ so that $\ell$ satisfies this constraint, thus verifying \eqref{eq:var}.
Verifying \eqref{eq:size-Qi} is immediate since $\ell \ge 1$ and we can make $c_2$ smaller than $(1-b)\delta\gamma^2/8$. 
To verify \eqref{eq:Di-vs-D}
 note that the map $t \mapsto \frac{\log_+(t)}{t^2}$ is decreasing when $t \geq e$. Since we can take $C_2 \ge e$ and $D_i \ge C_2\ell/u$ we get 
\[
\frac{(\overline{D}u/L)^2}{\log_+(\overline{D}u/L)} \leq \frac{(D_iu/L)^2}{\log_+(D_iu/L)} \implies (1/D_i)^2\log\lrpar{\frac{D_iu}{\ell}} \le (1/\overline{D})^2\log_+\lrpar{\frac{\overline{D}u}{\ell}}
\]
Finally, we verify \eqref{eq:D_-vs-D}. Since $d \leq \lambda n/\log(n)$ we have that $\overline{D} \geq n^{\gamma^2/(C\lambda)}$. Since $\log(1+x) > \min(1,x/2)$ for $x > 0$ we have that 
\[
\overline{D}^2\log_+\lrpar{\underline{D}u/\ell} \geq n^{\gamma^2/(C\lambda^2)} \cdot \min \lrpar{1, C_2u/\ell} \gsim n^{\gamma^2/(C\lambda^2)}. 
\]
Since $\underline{D}^2\log_+(\overline{D}u/\ell) \le (\ell/u + C_2)^2\gamma^2n^2/\ell^2 \ll n^{2.5}$ and $n^{2.5} \ll n^3$ we may take $\lambda < \gamma/\sqrt{C/3}$ so that $\overline{D}^2\log_+(\underline{D}u/\ell) \ge \underline{D}^2\log_+(\overline{D}u/\ell)$. Rearranging this inequality, we conclude that \eqref{eq:D_-vs-D} is satisfied. Having verified the necessary equations the result follows.
\end{proof}
\subsection{Proving the distance theorem}
 We now prove the distance theorem as a quick application of \refC{cor:sbp-for-projection} and \refT{thm:distance:3}
\begin{proof}[Proof of \refT{thm:distance}]
 Since $H^\perp$ lies in the nullspace of $A^\top$ by \refL{lem:comp}, with failure probability at most $e^{-cn}$, all unit vectors in $H^\perp$ are incompressible. 
By \refT{thm:distance:3} there exists choices of $u,\eta,\gamma$ such that 
\[
\Pr[\exists x \in\incomp(\delta,\rho)~s.t.~ A^\top x = 0 \text{ and } \logRLCD{X}{L}{u}{x} \leq c_1\sqrt{d}e^{c_2n/d}] \leq e^{-cn},
\]
where $L \in [2\sqrt{d},C_1\sqrt{D}]$.
In particular every vector $x \in H^\perp \cap \bS^{n-1}$ satisfies $\logRLCD{X}{L}{u}{x} > c_1\sqrt{d}e^{c_2n}$. By inspection of \eqref{eq:RLCD:subspace} from \refD{RLCD:def} it follows that $\logRLCD{X}{L}{u}{H^\perp} > c_1\sqrt{d}e^{c_2n}$. Since $2L^2 \ge d+2$ by \refC{cor:sbp-for-projection} it follows that 
\begin{align*}
\cL\lrpar{P_{H^\perp}X,t\sqrt{d}}
&\leq e^{-cn} + \lrpar{\frac{CLu}{\sqrt{d}}}^d\max\{t,\frac{\sqrt{d}}{c_1\sqrt{d}e^{-c_2n/d}}\}^d, \\
&\leq e^{-cn} + C^d(t^d + e^{-cn}),  \\
&\le e^{-cn} + (Ct)^d.
\end{align*}
\end{proof}
\begin{remark}\label{rem:discussion}
In the distance theorem for rectangular matrices that appears in \cite{GL} one can take $A \in \R^{n \times (n-d)}$ such that $d = O(n)$. Recall that \refT{thm:distance} requires that $d = O(n/\log n)$. The stronger restriction on $d$ is used for the union bound argument in the proof of \refP{prop:distance:2} and used to prove \eqref{eq:D_-vs-D} in the proof of \refT{thm:distance:3}. The restriction in the proof of \refT{thm:distance:3} can be overcome if one is more careful in bounding the terms in \eqref{eq:D_-vs-D} and assumes a stronger second moment constraint ($\Var|X| \lsim n$) or if one can prove a lower bound on Randomized Logarithmic LCD stronger than \eqref{eq:RLCD-lb-2} (the lower bound should be $(1+c)L/u$ where $c > 0$).  The more serious bottleneck is from \refP{prop:distance:2}. The issue is that the bound on the size of the net in \refP{prop:distance:2}, $(\overline{D}/U)^n$, will not be sufficiently small. To see this note that in order to prove \refP{prop:distance:2} using the net bound it is necessary that $\overline{D}^dU^{-n}$ be exponentially small (say at most $c^{-n}$). Therefore $d \leq n\log(cU)/\log(\overline{D})$. Since $\overline{D} = (2\sqrt{d}/u)\exp(\gamma n/d)$ this condition is weaker than $d \leq n\log(U)/\log(d) \implies d\log(d) \leq n\log(U)$, which fails to be true once $d$ is much larger than $n/\log(n)$. It seems that one would require a stronger result than \refT{thm:unstructured} or an alternative way to bound the net size. We note that in \cite{RV} the net used for their LCD variant \eqref{eq:LCD-og} has size of order $(\overline{D}/\sqrt{n})^n$. Such a bound for our net would be sufficient.
\end{remark}
As mentioned earlier, even without the distance theorem beyond this range one can recover smallest singular value estimates for inhomogenous matrices when $\lambda n/\log(n) \leq d \leq cn$. This fact will appear in a follow up work and the proof strategy is similar in spirit to the one employed in \cite{Litvak} for bounding the smallest singular value of rectangular subgaussian matrices of the same aspect ratio. In this way the distance theorem proven here is sufficient to recover smallest singular value estimates for inhomogenous rectangular matrices of all aspect ratios.
\bibliographystyle{plain}

\end{document}